\documentclass{article}

\usepackage{amsmath}
\usepackage{amssymb}
\usepackage{amsthm}
\usepackage{bm}
\usepackage{mathrsfs}
\usepackage{graphicx}
\usepackage{tikz}
\usetikzlibrary{arrows, arrows.meta}
\usepackage{nicefrac}
\usepackage{wrapfig}
\usepackage[numbers]{natbib}
\usepackage{mathtools}

\usepackage{xcolor}
\definecolor{dkblue}{cmyk}{1,.54,.04,.19} 

\usepackage{hyperref} 

\hypersetup{
  bookmarks=true,         
    unicode=false,          
    pdftoolbar=true,        
    pdfmenubar=true,        
    pdffitwindow=false,     
    pdfstartview={FitH},    
    pdftitle={Improved Regret for Bandit Convex Optimisation},    
    pdfauthor={Lattimore},     
    pdfsubject={Bandits},   
    pdfcreator={pdflatex},   
    pdfproducer={Producer}, 
    pdfkeywords={bandits} {statistics} {machine learning}, 
    pdfnewwindow=true,      
    colorlinks=true,        
    linkcolor=black,       
    citecolor=dkblue,       
    filecolor=dkblue,       
    urlcolor=dkblue,        
}

\usepackage[textsize=tiny]{todonotes}

\usepackage[capitalise]{cleveref}
\usepackage[bf]{caption}
\theoremstyle{plain}
\newtheorem{theorem}{Theorem}

\newtheorem{lemma}[theorem]{Lemma}
\newtheorem{proposition}[theorem]{Proposition}
\newtheorem{corollary}[theorem]{Corollary}

\theoremstyle{definition}

\newtheorem{remark}[theorem]{Remark}
\theoremstyle{remark}

\newcommand{\cK}{\mathcal K}
\newcommand{\cD}{\mathcal D}
\newcommand{\cF}{\mathscr F}
\newcommand{\cE}{\mathcal E}

\newcommand{\R}{\mathbb R}
\newcommand{\norm}[1]{\vert #1 \vert}
\newcommand{\ip}[1]{\langle #1 \rangle}
\newcommand{\argmin}{\operatorname{arg\,min}}
\newcommand{\argmax}{\operatornamewithlimits{arg\,max}}
\renewcommand{\d}[1]{\operatorname{d}\!#1}
\newcommand{\aff}{\operatorname{aff}}
\newcommand{\E}{\mathbb E}
\newcommand{\Reg}{\mathfrak{R}}
\newcommand{\BReg}{\mathfrak{BR}}
\newcommand{\ceil}[1]{\left\lceil {#1} \right\rceil}

\newcommand{\vol}{\operatorname{vol}}

\newcommand{\zero}{\bm{0}}
\newcommand{\diam}{\operatorname{diam}}
\newcommand{\dist}{\operatorname{dist}}
\newcommand{\cG}{\mathcal G}
\newcommand{\cJ}{\mathcal J}
\newcommand{\cC}{\mathcal C}
\newcommand{\const}{{\normalfont \textrm{const}}}
\newcommand{\sind}{\bm{1}}
\newcommand{\bbP}{\mathbb P}
\newcommand{\Psip}{\tilde \Psi}

\begin{document}

\title{\large \textsc{Improved Regret for Zeroth-Order Adversarial Bandit Convex Optimisation}}
\author{\large Tor Lattimore \\[0.3cm] DeepMind, London \\ {\tt lattimore@google.com}}
\date{}

\maketitle

\noindent \textbf{Abstract}\,\, We prove that the information-theoretic upper bound on the minimax regret 
for zeroth-order adversarial bandit convex optimisation is at most $O(d^{2.5} \sqrt{n} \log(n))$, where $d$ is the dimension and $n$ is the number
of interactions. This improves on the bound of $O(d^{9.5} \sqrt{n} \log(n)^{7.5})$ by Bubeck et al.\ (2017).
The proof is based on identifying an improved exploratory distribution for convex functions. \\[-0.3cm]



\section{Introduction}

Let $\cK \subset \R^d$ be a convex body (convex, compact with non-empty interior) and $\cG$ be a set of convex functions from $\cK$ to $[0,1]$.
At the start of the game, an adversary secretly chooses a sequence $(f_t)_{t=1}^n$ with $f_t \in \cG$.
Then, in each round $t$, the learner chooses an action $x_t \in \cK$, possibly with randomisation, and observes only the loss $f_t(x_t)$.
The minimax regret over $n$ rounds is
\begin{align}
\Reg_n^\star(\cG) = \inf_{\text{policies}} \sup_{(f_t)_{t=1}^n \in \cG^n} \max_{x \in \cK} \E\left[\sum_{t=1}^n f_t(x_t) - f_t(x)\right]\,,
\label{eq:regret}
\end{align}
where the inf is over all policies of the learner that determine the actions $(x_t)_{t=1}^n$.
The expectation integrates over the randomness of the actions $(x_t)_{t=1}^n$.
Our contribution is a proof of the following theorem.

\begin{theorem}\label{thm:regret}
Suppose that $\cK$ contains a unit-radius Euclidean ball and $\cG$ is the set of all convex functions from $\cK$ to $[0,1]$. Then
\begin{align*}
\Reg_n^\star(\cG) \leq \const \cdot d^{2.5} \sqrt{n} \log\left(n \diam(\cK)\right)\,,
\end{align*}
where $\const$ is a universal constant and $\diam(\cK) = \max_{x, y \in \cK} |x - y|$ is the diameter of $\cK$ with $|\cdot|$ the standard Euclidean norm.
\end{theorem}

The assumption that $\cK$ contains a unit-radius Euclidean ball is relatively mild.
Any convex body can be rescaled to contain a unit ball, while the regret depends only logarithmically on the diameter.
As in previous work, we make use of a simple reduction that allows us to restrict slightly the class of functions available to the adversary \citep{BE18,BLE17}.
Define a constant $m = 1/((n+1) \diam(\cK)^2)$ and let $\cF$ be the space of all convex functions $\cK$ to $[0,1]$ that are
\begin{itemize}
\item[(a)] $n$-Lipschitz: $f(x) - f(y) \leq n \norm{x - y}$ for all $x, y \in \cK$; and
\item[(b)] $m$-strongly convex: for all $x, y \in \cK$ and $\lambda \in [0,1]$,
\begin{align*}
\!\!\!\!\!\!
f(\lambda x + (1 - \lambda) y) \leq \lambda f(x) + (1 - \lambda) f(y) - \frac{1}{2} m \lambda(1-\lambda) \norm{x - y}^2\,. 
\end{align*}
\end{itemize}
\cref{prop:reduce} in \cref{sec:reduction} shows that it suffices to prove \cref{thm:regret} with $\cF$ rather than $\cG$.
Briefly, the reduction works by showing that bounded convex functions from $\cK$ to $[0,1]$ cannot have large directional derivatives except close to the boundary $\partial \cK$ and must have near-minimisers that
are not too close to the boundary. This means the learner can play on a subset of $\cK$ on which the loss functions are $n$-Lipschitz without sacrificing much in terms of the regret. 
Strong convexity is introduced by adding a small quadratic to all loss functions, which has only a negligible impact because $m$ is small.

The next (more or less known) theorem serves as our starting point. It follows from the machinery developed by \citeauthor{BDKP15} \cite{BDKP15}, \citeauthor{BE18} \cite{BE18} and \citeauthor{RV14} \cite{RV14}.
For completeness, an outline of the proof is given in \cref{app:ids}.

\begin{theorem}\label{thm:ids}
Let $\alpha, \beta \in \R$ be non-negative and for $f \in \cF$, let $f_\star = \min_{x \in \cK} f(x)$.
Suppose that for any $\bar f \in \cF$ and finitely supported distribution $\mu$ on $\cF$ with the discrete $\sigma$-algebra there exists a probability measure $\rho$ on $\cK$ such that
\begin{align}
\int_\cK \bar f(x) \d{\rho}(x) - \int_\cF f_\star \d{\mu}(f) \leq
\alpha + \sqrt{\beta \int_{\cF} \int_{\cK} (\bar f(x) - f(x))^2 \d{\rho}(x) \d{\mu}(f)}\,.
\label{eq:inf}
\end{align}
Then the minimax regret is bounded by
\begin{align*}
\Reg_n^\star(\cF) \leq 3 + n \alpha + \sqrt{\beta d n \log\left(3n^2 \max(1, (2\beta)^{1/2}) \diam(\cK)\right)}\,.
\end{align*}
\end{theorem}

The distribution $\rho$ in \cref{thm:ids} is called an exploratory distribution.
\citeauthor{BE18} \cite{BE18} established the conditions of \cref{thm:ids} with
$\alpha = 1/\sqrt{n}$ and $\beta = d^{21} \operatorname{polylog}(n)$. 
The next theorem improves on this result.

\begin{theorem}\label{thm:main}
For any $\bar f \in \cF$ and finitely supported distribution $\mu$ over $\cF$, there exists a probability measure $\rho$ on $\cK$ such that \cref{eq:inf} holds
with $\alpha = 1/n$ and $\beta = \const \cdot d^4 \log(nd/m)$, where $\const$ is a universal constant.
\end{theorem}

Combining \cref{thm:ids,thm:main} with the reduction in \cref{prop:reduce} proves \cref{thm:regret}.

\paragraph{Related work}
Online convex optimisation is usually studied under the assumption that the learner has access to the (sub-)gradient $\nabla f_t(X_t)$, or even the whole function $f_t$.
A number of perspectives on this vast literature can be found in recent books and notes by \citeauthor{Ces06} \cite{Ces06}, \citeauthor{Haz16} \cite{Haz16} and \citeauthor{Ora19}~\cite{Ora19}.
There is far less work in the zeroth-order bandit setting in which the learner only has access to point evaluations.
A natural idea is to use importance-weighting to estimate the gradients. At least with current tools, however, the resulting bias/variance tradeoff leads to suboptimal regret \cite{HPGS16}.

The function class $\cF$ is omitted from the following regret bounds, to emphasise that the assumptions vary in minor ways.
Constant factors are also omitted, while Big-O is used to hide dependence on other parameters such as smoothness/dimension/strong convexity.
Information-theoretic means were used by \citeauthor{BDKP15} \cite{BDKP15}
to show that the minimax regret is $\Reg_n^\star \leq \sqrt{n} \log(n) $ when $d = 1$.
The multi-dimensional problem was considered by 
\citeauthor{HY16} \cite{HY16}, who showed that the minimax regret is $O(\sqrt{n} \operatorname{polylog}(n))$, but with an exponential dependence on the dimension.
Shortly after, \citeauthor{BE18} \cite{BE18} generalised the information-theoretic machinery to prove that $\Reg_n^\star \leq d^{11} \sqrt{n} \log(n)^4 $, breaking the exponential dependence
on the dimension while retaining square root dependence on the horizon.
None of these works provide an efficient algorithm.
More recently, \citeauthor{BLE17} \cite{BLE17} used kernel-based estimators and tools from online convex optimisation to show that $\Reg_n^\star \leq d^{9.5} \sqrt{n} \log(n)^{7.5}$.
Furthermore, their algorithm can be implemented in polynomial time (with reasonable assumptions) with the price that the dimension-dependence increases to $d^{10.5}$.
They conjectured that the optimal regret is $\Reg_n^\star \leq d^{1.5} \sqrt{n} \operatorname{polylog}(n)$. The best known lower bound is $\Reg_n^\star \geq d \sqrt{n}$,
which holds even when the function class is restricted to linear functions \cite{DHK08}. There is also a line of work that exploits strong convexity to obtain better bounds.
In particular, \citeauthor{HL14} \cite{HL14} show that $\Reg_n^\star \leq d  (d + \beta/m)^{1/2} \sqrt{n} \log(n)^{1/2}$ for $m$-strongly convex and $\beta$-smooth functions.
Even setting aside the smoothness assumption, the polynomial dependence on the strong convexity parameter means that enforcing strong convexity 
by adding a small quadratic to the losses blows up the bound and leads to suboptimal rates.

\paragraph{Comparison to \citeauthor{BE18} \cite{BE18}}
Both works rely on minimax duality to relate the Bayesian and frequentist regret, and on the information-theoretic tools developed by \citeauthor{RV14} \cite{RV14} connecting the Bayesian regret
to a tradeoff between regret and mutual information about the optimal action. What is entirely different is how the aforementioned tradeoff is approximately optimised.
The key property exploited by \citeauthor{BE18} is that suitably bounded convex functions must be approximately linear on many reasonably large balls, which is
combined with an elementary `line of sight' argument to construct an exploratory distribution. The main challenge is to establish the approximate linearity, which can be viewed 
as a quantification of Alexandrov's theorem that convex functions are almost everywhere twice differentiable. In this sense, the structure they exploit depends on a deep understanding of the behaviour
of the gradients of convex functions. As explained momentarily, the structure exploited in the present work concerns the level sets of convex functions, which are less sophisticated
and more well-studied. It is important to emphasise that both results rely on versions of \cref{thm:ids}, which is non-constructively established via minimax duality. 
Because of this, neither approach yields an algorithm.

\paragraph{Preliminaries}
The $d$-dimensional sphere is $S^d = \{x \in \R^{d+1} : \norm{x} = 1\}$.
The Minkowski sum of sets $A$ and $B$ is denoted $A + B$. When $x$ is a vector, $A + x = A + \{x\}$.
Let $\vol_p$ denote the $p$-dimensional Hausdorff measure on $\R^d$, normalised to coincide with the Lebesgue measure.
Given $x, y \in \R^d$, $[x,y] = \{t x + (1 - t) y : t \in [0,1]\}$ is the chord connecting $x$ and $y$.
The sets $[x, y)$ and $(x, y]$ and $(x, y)$ are defined similarly but with appropriate end-points removed.
When $x \neq \zero$, the hyperplane with normal $x$ is $x^\perp = \{y \in \R^d : \ip{x, y} = 0\}$ and $P_x(y) = \argmin_{z \in x^\perp} |y - z|$ is the Euclidean projection of $y$ onto $x^\perp$.
The shadow of a convex body $K \subset \R^d$ in direction $x$ is $P_x(K) = \{P_x(y) : y \in K\}$.
The infimum of a convex function $f : K \to \R$ is $f_\star = \inf_{x \in K} f(x)$.
The boundary of $K$ is $\partial K$.

\section{Combining exploratory distributions}\label{sec:combine}
The plan is to establish the conditions of \cref{thm:ids} for suitable values of $\alpha$ and $\beta$, which means that for any $\bar f \in \cF$ and distribution $\mu$ on $\cF$ we need to find a probability measure $\rho$ on $\cK$ satisfying \cref{eq:inf}.
To make the problem more manageable, we first prove that exploratory distributions can be combined.

\begin{lemma}\label{lem:combine}
Let $\bar f \in \cF$ and $\cF = \cup_{i=1}^k \cF_i$. Assume there exist probability measures $(\rho_i)_{i=1}^k$ on $\cK$ such that
for all $i \in \{1,\ldots,k\}$,
\begin{align}
\int_\cK \bar f(x) \d{\rho_i(x)} - f_\star \leq \alpha + \sqrt{\beta \int_\cK (\bar f(x) - f(x))^2 \d{\rho_i}(x)}  \quad \text{ for all } f \in \cF_i\,.
\label{eq:ass}
\end{align}
Then, for any finitely supported distribution $\mu$ on $\cF$, there exists a probability measure $\rho$ on $\cK$ such that
\begin{align*}
\int_\cK \bar f(x) \d{\rho}(x) - \int_\cF f_\star \d{\mu}(f)
&\leq  \alpha + \sqrt{\beta k \int_\cF \int_\cK \left(\bar f(x) - f(x)\right)^2 \d{\rho}(x) \d{\mu}(f)}\,. 
\end{align*}
\end{lemma}

\begin{proof}
The argument is algebraically identical to that used by \citeauthor{RV16} \cite{RV16} to bound the information ratio for Thompson sampling.
Assume without loss of generality that $(\cF)_{i=1}^k$ are disjoint and $\mu(\cF_i) > 0$ for all $i \in \{1,\ldots,k\}$. 
Let $\mu_i$ be the probability measure obtained by conditioning $\mu$ on $\cF_i$: $\mu_i(A) = \mu(A \cap \cF_i) / \mu(\cF_i)$.
Then, letting $q_i = \mu(\cF_i)$ and $\rho = \sum_{i=1}^k q_i \rho_i$,
\begin{align*}
\int_\cK \bar f(x) \d{\rho}(x) &- \int_\cF f_\star \d{\mu}(f)
= \sum_{i=1}^k q_i \left(\int_\cK \bar f(x) \d{\rho_i}(x) - \int_\cF f_\star \d{\mu_i}(f)\right) \\
&\leq \alpha + \sum_{i=1}^k q_i \sqrt{\beta \int_\cF \int_\cK \left(\bar f(x) - f(x)\right)^2 \d{\rho_i}(x) \d{\mu_i}(f)} \\
&\leq \alpha + \sqrt{\beta k \sum_{i=1}^k q_i^2 \int_\cF \int_\cK \left(\bar f(x) - f(x)\right)^2 \d{\rho_i}(x) \d{\mu_i}(f)} \\
&\leq \alpha + \sqrt{\beta k \sum_{i, j=1}^k q_i q_j \int_\cF \int_\cK \left(\bar f(x) - f(x)\right)^2 \d{\rho_i}(x) \d{\mu_j}(f)} \\
&= \alpha + \sqrt{\beta k \int_\cF \int_\cK \left(\bar f(x) - f(x)\right)^2 \d{\rho}(x) \d{\mu}(f)}\,, 
\end{align*}
where the first inequality follows from the assumption in \cref{eq:ass} and Jensen's inequality and the second follows from Cauchy-Schwarz.
\end{proof}

\section{Proof of \cref{thm:main}}\label{sec:main}
Before the detailed calculations, we spend a moment to explain the main argument in dimension one.

By \cref{lem:combine}, it suffices to find a partition of $\cF = \cup_{i=1}^k \cF_i$ and corresponding exploratory distributions $(\rho_i)_{i=1}^k$ satisfying \cref{eq:ass}
with $k$ not too large. 
Assume without loss of generality that $\bar f$ is minimised at $0 \in \cK$ and define level set $K_\delta = \{x \in \cK : \bar f(x) \leq \bar f_\star + \delta\}$.
Next, let $\epsilon$ be such that $K_\epsilon \subset [-1/(2n^2), 1/(2n^2)]$, which by $m$-strong convexity is possible 
with $\epsilon = 1/\operatorname{poly}(n)$. For logarithmically large $k$ let 
$\cE = \{0, \epsilon, 2\epsilon, \ldots, 2^{k-1} \epsilon\}$ and $\rho_\epsilon$
be the uniform distribution on $\partial K_\epsilon$, which is either a Dirac distribution ($\epsilon = 0$) or a uniform mixture of two Dirac distributions ($\epsilon > 0$).
The situation with $k = 5$ is illustrated in \cref{fig:approach}.
Let $f \in \cF$ be minimised at $x_\star \in \cK$ and consider three cases.
First, when $f_\star \geq \bar f_\star - 1/n$, then \cref{eq:ass} holds trivially with $\alpha = 1/n$ and $\beta = 0$ for $f$ and $\rho_0$.
On the other hand, if $f_\star < \bar f_\star - 1/n$ and $x_\star \in K_\epsilon$, then the fact that $f$ is $n$-Lipschitz shows that
$\bar f(0) - f(0) \geq \bar f(0) - f_\star - n \norm{x_\star} \geq \bar f(0) - f_\star - 1/(2n) \geq \frac{1}{2}(\bar f(0) - f_\star)$
and \cref{eq:ass} holds for $f$ and $\rho_0$ with $\alpha = 0$ and $\beta = 4$.
Otherwise $f_\star \leq \bar f_\star$ and there exists a largest $\delta \in \cE$ such that $x_\star \notin K_\delta$. \cref{fig:approach} highlights the
relevant level set in red and shows that \cref{eq:ass} holds for $f$ and $\rho_\delta$ with $\alpha = 0$ and $\beta$ a universal constant. 
This last fact made formal in \cref{lem:los}.
Hence, $(\rho_\delta)_{\delta \in \cE}$ satisfies the conditions of \cref{lem:combine} with $\alpha = 1/n$ and $\beta$ a universal constant. 

\begin{center}
\includegraphics[width=9cm]{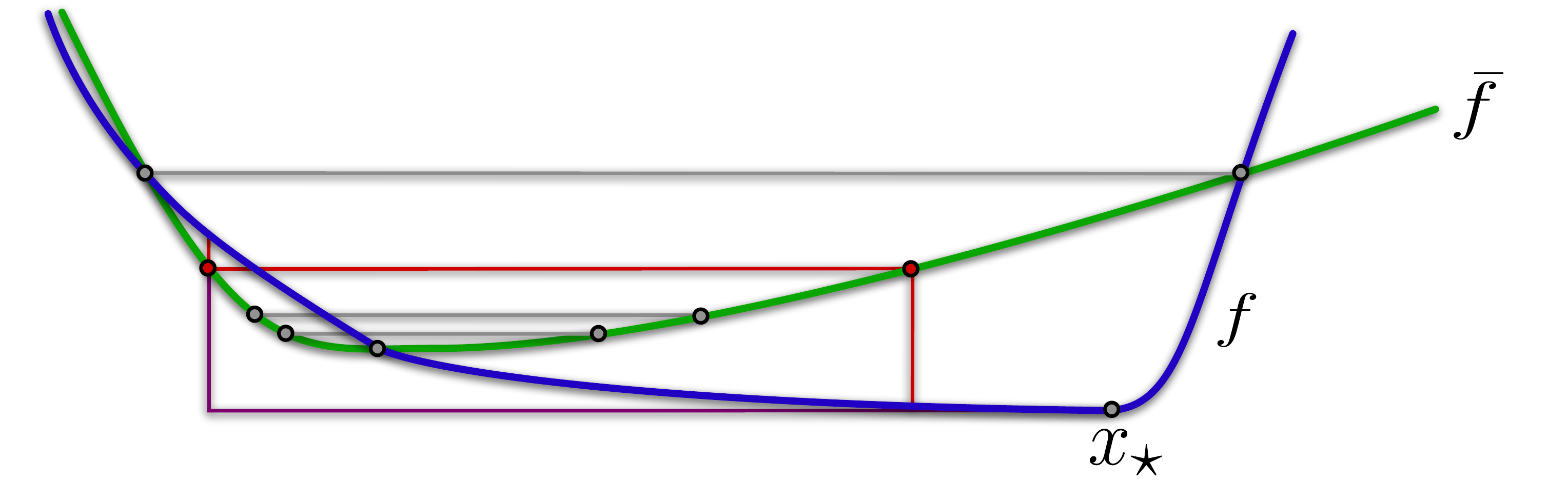}
\captionof{figure}{
}\label{fig:approach}
\end{center}

Superficially, not much changes in higher dimensions. Exploratory distributions are still constructed as probability measures on the boundaries of level sets of $\bar f$ and combined
using \cref{lem:combine}.
Unfortunately, however, the uniform probability measure turns out to be a poor choice for level sets that are not suitably regular, which necessitates a more complicated choice.
Besides this, the geometric grid on level sets needs to be finer and the precise location of $x_\star$ relative to the level sets
becomes more important.

The first lemma shows that convex functions with different minimisers must differ along suitably chosen rays.
The situation is illustrated in \cref{fig:los2}.

\begin{lemma}\label{lem:los}
Let $\cD \subset \R^d$ be a convex body and $f, g : \cD \to \R$ be convex functions with $f$ minimised at $x_\star \in \cD$ and $f_\star \leq g_\star$. Suppose that $\epsilon > 0$ and $x_\star \notin K = \{x \in \cD : g(x) \leq g_{\star} + \epsilon\}$.
Assume that $x, y \in \partial K$ and $x = \psi x_\star + (1 - \psi) y$ with $\psi \in (0, 1)$. Then,
\begin{align*}
(f(x) - g(x))^2 + (f(y) - g(y))^2 \geq \frac{1}{2} \psi^2 \left(g_\star + \epsilon - f_\star)\right)^2\,.
\end{align*}
\end{lemma}

\begin{proof}
The definition of the level set and continuity of $g$ means that $g(x) = g_\star + \epsilon$. Furthermore, $g(y) \leq g_\star + \epsilon$, with a strict inequality
only possible if $y$ is on the boundary of $\cD$.
Suppose that $\psi f_\star + (1 - \psi) f(y) \geq g_\star + \epsilon$.  
Then 
\begin{align*}
(f(y) - g(y))^2 \geq \left(\frac{\psi}{1 - \psi}\right)^2 \left(g_\star + \epsilon - f_\star\right)^2 \geq \frac{1}{2} \psi^2 (g_\star + \epsilon - f_\star)^2\,.
\end{align*}
Otherwise, if $\psi f_\star + (1 - \psi) f(y) < g_\star + \epsilon$, then,
\begin{align*}
(g(y) - f(y))^2 + (g(x) - f(x))^2
&= (g(y) - f(y))^2 + (g_\star + \epsilon - f(x))^2 \\
&\geq (g(y) - f(y))^2 + (g_\star + \epsilon - \psi f_\star - (1 - \psi) f(y))^2 \\
&\geq \min_{a \in \R}\left[ (g(y) - a)^2 + (g_\star + \epsilon - \psi f_\star - (1 - \psi) a)^2\right] \\
&= \frac{\left(\psi(g_\star + \epsilon - f_\star) + (1 - \psi) (g_\star + \epsilon - g(y))\right)^2}{2 - 2\psi + \psi^2} \\
&\geq \frac{1}{2} \psi^2 (g_\star + \epsilon - f_\star)^2\,,
\end{align*}
where the first inequality follows from convexity of $f$ implying that 
$f(x) \leq \psi f_\star + (1 - \psi) f(y) < g_\star + \epsilon$.
The last inequality follows from naive bounding.
\end{proof}

\begin{center}
\includegraphics[width=6.75cm]{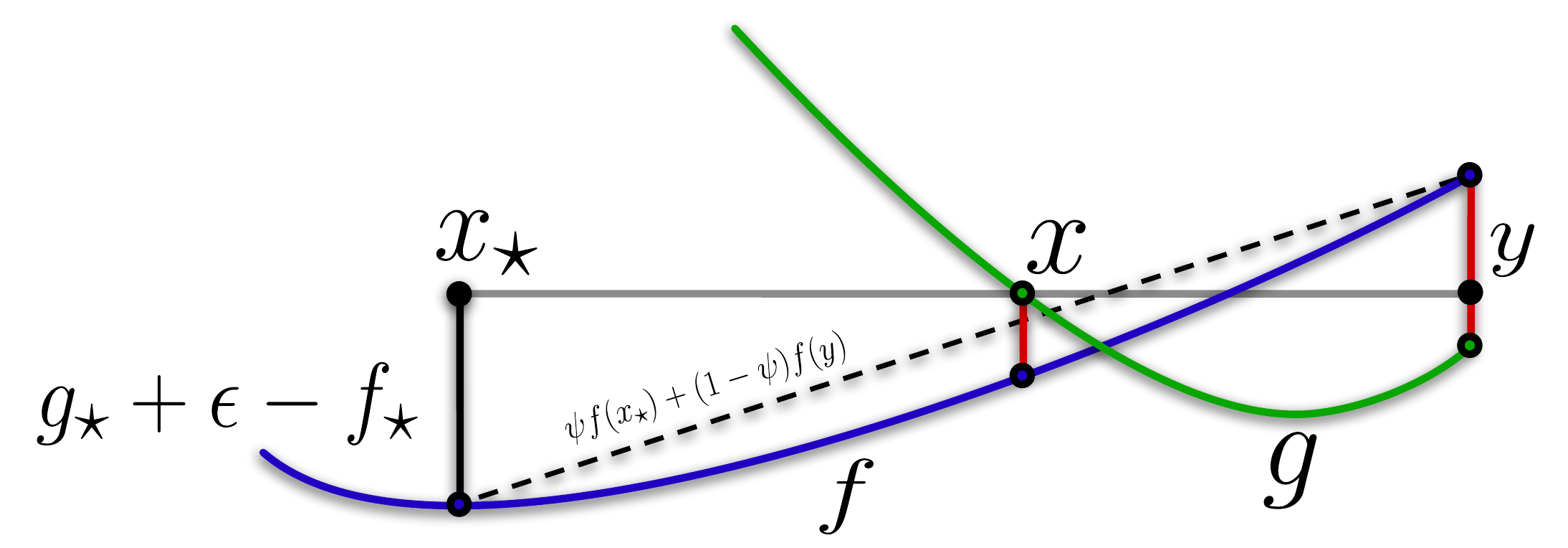}
\hspace{0.75cm}
\includegraphics[width=4cm]{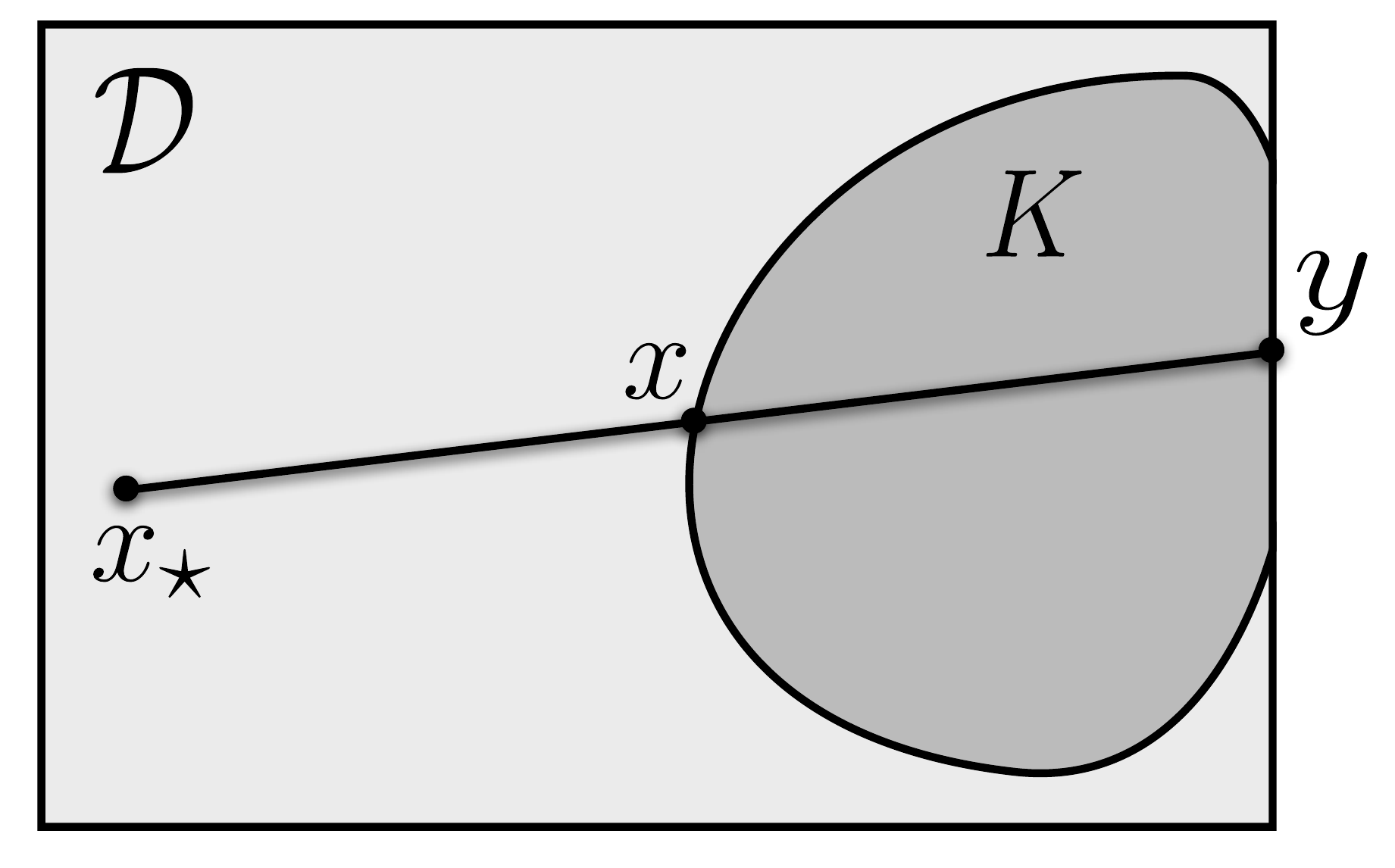}
\captionof{figure}{In the figure on the left, the horizontal grey line marks $g_\star + \epsilon = g(x)$. 
The proof of \cref{lem:los} shows that at least one of the red vertical lines is roughly $\psi$ times the length of the black line. 
The figure on the right shows the top down view.
}
\label{fig:los2}
\end{center}

The next lemma shows that sampling uniformly from the boundary of a level set is a good exploratory distribution for functions $f \in \cF$ minimised at $x_\star$ provided
two conditions hold. First,
the level set should have large shadows relative to its surface area. Second, the `average' depth/length ratio of rays emanating from $x_\star$ and passing 
through the level set should be approximately $1/d$, which is sufficiently large that \cref{lem:los} has teeth and sufficiently small that rays emanating from $x_\star$ and passing through the level
set are approximately parallel. This last property is essential when relating rays, shadows and surface area, and is exploited in the second step of the proof of \cref{lem:key}.
\cref{lem:combine} will then be used to combine exploratory distributions produced from different level sets.
Level sets with small shadows are handled by transforming coordinates so that the level set is in minimal surface area position (definition in \cref{rem:key}), and then pulling
back the uniform distribution on the transformed level set.

Before the statement, we define a measure of the average depth/length ratio for rays emanating from a point and passing through a convex body.
The concepts are illustrated in \cref{fig:shadow}.
Let $K \subset \R^d$ be a convex body with $\zero \in K$ and $x \notin K$.
Let $\pi_K(x, \cdot) : P_x(K) \to \partial K$ be an `inverse' of the projection $P_x$ defined by
\begin{align*}
\pi_K(x, z) = \argmin_{y \in K \cap \{z + tx : t \in \R\}} \ip{y, x} \,.
\end{align*}
For $x \in \R^d$ and $z \in P_x(K)$, define the depth/distance ratio by
\begin{align*}
\Psi_K(x, z) = \frac{\vol_1(K \cap [x, \pi_K(x, z)])}{\vol_1([x, \pi_K(x, z)])} \,.
\end{align*}
The average depth/distance ratio with respect to the uniform probability measure on $P_x(K)$ is
\begin{align}
\Psi_K(x) = \frac{1}{\vol_{d-1}(P_x(K))} \int_{P_x(K)} \Psi_K(x, z) \d{\vol_{d-1}}(z)\,.
\label{def:psi}
\end{align}
The maximum depth/distance ratio is $\Psi_K^\infty(x) = \max_{z \in P_x(K)} \Psi_K(x, z)$.
A number of properties of $\Psi$ are collected in \cref{sec:technical}. The most important are that 
$z \mapsto \Psi_K(x, z)$ is concave on its domain $P_x(K)$ and $\Psi_{TK}(Tx) = \Psi_{K}(x)$ for any linear bijection $T : \R^d \to \R^d$.

\begin{center}
\includegraphics[width=7cm]{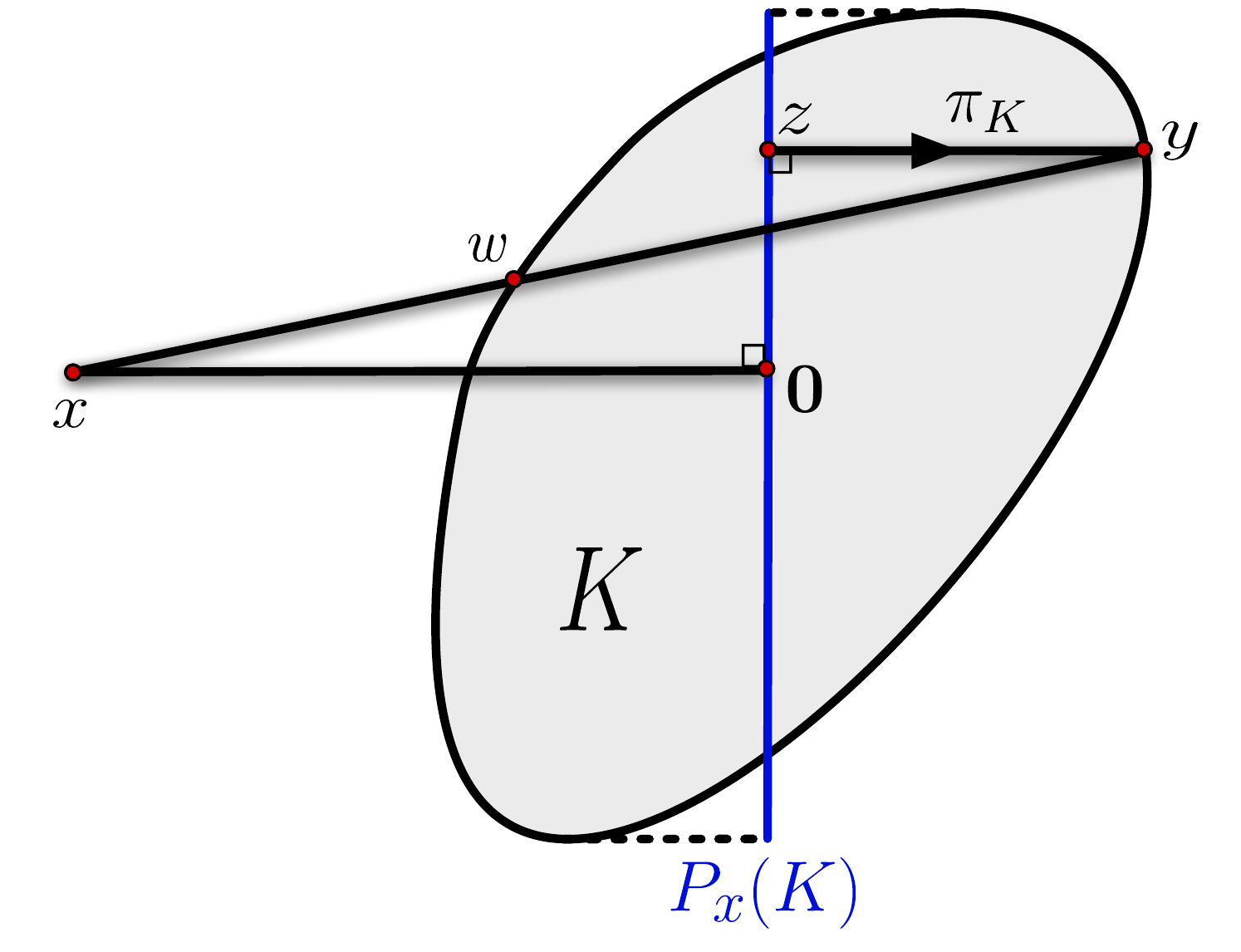}
\captionof{figure}{
For the configuration in the figure, $\pi_K(x, z) = y$ and $\Psi_K(x, z) = \norm{y - w} / \norm{y - x}$, which is the ratio of the depth of $K$ along the chord $[x, y]$ and the length 
of the chord. The quantity $\Psi_K(x)$ averages $\Psi_K(x, z)$ with respect to the uniform probability measure over all $z$ in the shadow $P_x(K)$.}
\label{fig:shadow}
\end{center}

\begin{lemma}\label{lem:key}
Let $\cD \subset \R^d$ be a convex body and
$f, g : \cD \to \R$ be convex with $f_\star \leq g_\star$, $\epsilon > 0$ and $K = \{x \in \cD : g(x) \leq g_\star + \epsilon\}$.
Let $x_\star \in \cD$ be the minimiser of $f$ and
assume that $\zero \in K$ and $\Psi_K(x_\star) \in [1/(128d), 1/(32d)]$.
Then, 
\begin{align*}
\int_{\partial K} g \d{\rho} - f_\star 
&\leq  2^{13} d\sqrt{ \max_{\theta \in S^{d-1}} \left(\frac{\vol_{d-1}(\partial K)}{\vol_{d-1}(P_{\theta}(K))}\right) \int_{\partial K} (f - g)^2 \d{\rho}}\,. \qedhere
\end{align*}
where $\rho = \vol_{d-1} / \vol_{d-1}(\partial K)$ is the surface area probability measure on $\partial K$. 
\end{lemma}

\begin{proof}[Proof of \cref{lem:key}]
To reduce clutter, let $\Psip(z) = \Psi_K(x_\star, z)$ and abbreviate $\pi(z) = \pi_K(x_\star, z)$ and $P = P_{x_\star}$.
The rest of the proof is divided into three steps. The first uses \cref{lem:concave} to show that a large
fraction of rays through $K$ from $x_\star$ are reasonably deep.
The second step connects surface integrals over $\partial K$ to rays cutting $K$ from $x_\star$ and the third puts together
the pieces using \cref{lem:los}.

\paragraph{Step 1: Preponderance of deep cuts}
Let $B \subset P(K)$ be the set given by
\begin{align*}
B = \left\{z \in P(K) : \frac{\Psip(z)}{\frac{1}{\vol_{d-1}(P(K))} \int_{P(K)} \Psip \d{\vol_{d-1}}} \in \left[1/4, 16\right]\right\}\,,
\end{align*}
which is the set of points in $z \in P(K)$ where $\Psip(z)$ is close to its mean $\Psi_K(x_\star)$.
By \cref{lem:concave} and concavity of $\Psip$ (\cref{lem:psi-concave}), 
\begin{align}
\frac{\vol_{d-1}(B)}{\vol_{d-1}(P(K))} \geq \frac{1}{32}\,.
\label{eq:A}
\end{align}
Furthermore, by the assumptions in the lemma,
\begin{align*}
\Psi_K(x_\star) = \frac{1}{\vol_{d-1}(P(K))} \int_{P(K)} \Psip \d{\vol_{d-1}} \in \left[\frac{1}{128d}, \frac{1}{32d}\right]\,,
\end{align*}
which implies that for $z \in B$,
\begin{align}
\Psip(z) \in \left[\frac{2^{-9}}{d}, \frac{1}{2d}\right]\,.
\label{eq:Psi}
\end{align}

\paragraph{Step 2: Surface area to rays}
Let $C = \pi(B) \subset \partial K$. Define a function $\kappa : C \to \partial K$ so that $[x_\star, y] \cap K = [\kappa(y), y]$, which is the point $w$ in \cref{fig:shadow} and
is chosen so that for any $z \in P(K)$,
\begin{align}
\kappa(\pi(z)) = \Psip(z) x_\star + (1 - \Psip(z)) \pi(z) \,.
\label{eq:kappa}
\end{align}
Let $D = \kappa(C)$.
The goal in this step is to establish \cref{eq:proj} below, which makes the connection between rays and surface area.
Let $\varphi : \cD \to [0, \infty)$ be measurable. Then the following holds:
\begin{align}
\int_{C} \varphi \d{\vol_{d-1}} &\geq \int_{B} \varphi \circ \pi \d{\vol_{d-1}}  \nonumber \\ 
\int_{D} \varphi \d{\vol_{d-1}} &\geq \frac{1}{4} \int_{B} \varphi \circ \kappa \circ \pi \d{\vol_{d-1}} \,.
\label{eq:proj}
\end{align}
The first inequality is true because the projection of $\partial K$ onto $P(K)$ only decreases surface area.
The second inequality is immediate when $d = 1$, since $B, C$ and $D$ are singletons. Assume for the remainder of this step that $d > 1$.
Define $\lambda : P(K) \to \R$ and $\Lambda : P(K) \to P(K)$ by $\lambda(z) = 1 - \Psip(z)$ and $\Lambda(z) = \lambda(z) z$. 
Note, $\lambda(z) \in (0,1)$ ensures that $\Lambda(z) \in P(K)$ for all $z \in P(K)$. Also, concavity of $\Psip$ implies convexity of $\lambda$.
For the remainder of this step, choose coordinates on $P(K) \subset x_\star^\perp \cong \R^{d-1}$ via the obvious isometry to $\R^{d-1}$, which means that gradients and similar are written as $(d-1)$-dimensional vectors. 
For the second inequality in \cref{eq:proj}, differentiating $\Lambda$ at $z \in B$ yields $\operatorname{D} \Lambda(z) = \lambda(z) \operatorname{Id} + z \nabla \lambda(z)^\top$, which exists $\vol_{d-1}$-a.e.\ by convexity of $\lambda$.
Recalling the matrix determinant lemma: $\det(A + uv^\top) = \det(A) (1 + v^\top A^{-1}u)$, we have 
\begin{align}
\det(\operatorname{D} \Lambda(z)) 
= \lambda(z)^{d-1} \left(1 + \frac{\ip{\nabla \lambda(z), z}}{\lambda(z)}\right) 
&\geq \lambda(z)^{d-1}\left(2 - \frac{\lambda(\zero)}{\lambda(z)}\right) 
\geq \frac{1}{4} \,.
\label{eq:det}
\end{align}
The first inequality follows from convexity of $\lambda$, so that $\lambda(\zero) \geq \lambda(z) - \ip{\nabla \lambda(z), z}$. 
The last inequality follows from naive bounding because for $z \in B$, $\lambda(z) \geq 1-1/(2d)$ and $\lambda(\zero) \leq 1$.
The claim follows since, by the definition of $\Lambda$ and \cref{eq:kappa}, $P \circ \kappa \circ \pi = \Lambda$, which implies that
\begin{align*}
\int_D \varphi \d{\vol_{d-1}} 
&\geq \int_{\Lambda(B)} \varphi \circ \kappa \circ \pi \circ \Lambda^{-1} \d{\vol_{d-1}} 
\geq \frac{1}{4} \int_B \varphi \circ \kappa \circ \pi \d{\vol_{d-1}}\,,
\end{align*}
where in the first inequality we used the facts that $\Lambda(B) = P(D)$ and projections decrease surface area. The second inequality follows from \cref{eq:det}.

\paragraph{Step 3: Combining}
Let $z \in B$ and $y = \pi(z) \in \partial K$ and $w = \kappa(y) = \Psip(z) x_\star + (1 - \Psip(z)) y \in \partial K$.
Hence, by \cref{eq:Psi} and \cref{lem:los}, the following holds on $B$, 
\begin{align}
(f - g)^2 \circ \pi + (f - g)^2 \circ \kappa \circ \pi \geq \frac{1}{2} (g_\star + \epsilon - f_\star)^2 \Psip^2 \geq \frac{(g_\star + \epsilon - f_\star)^2}{2^{19} d^2}\,.
\label{eq:los-use}
\end{align}
Therefore, by \cref{eq:A}, \cref{eq:proj} and \cref{eq:los-use}, and recalling the definition of $\rho$ in the lemma statement, 
\begin{align*}
\int_{\partial K} (f-g)^2 \d{\rho}
&\geq \frac{1}{\vol_{d-1}(\partial K)} \int_{C \cup D} (f-g)^2 \d{\vol_{d-1}} \\
&\geq  \frac{1}{4 \vol_{d-1}(\partial K)} \int_B \left((f-g)^2 \circ \pi + (f-g)^2 \circ \kappa \circ \pi \right) \d{\vol_{d-1}} \\
&\geq \frac{(g_\star + \epsilon - f_\star)^2}{2^{21} d^2} \left(\frac{\vol_{d-1}(B)}{\vol_{d-1}(\partial K)}\right) \\
&\geq \frac{(g_\star + \epsilon - f_\star)^2}{2^{26} d^2} \min_{\theta \in S^{d-1}} \left(\frac{\vol_{d-1}(P_{\theta}(K))}{\vol_{d-1}(\partial K)}\right) \\
&\geq \frac{1}{2^{26} d^2} \left(\int_{\cK} g \d{\rho} - f_\star\right)^2 \min_{\theta \in S^{d-1}} \left(\frac{\vol_{d-1}(P_{\theta}(K))}{\vol_{d-1}(\partial K)}\right) \,,
\end{align*}
where the final inequality follows since $f_\star \leq g_\star$ by assumption and $g \leq g_\star + \epsilon$ on $\partial K$.
Rearranging yields the result.
\end{proof}

\begin{remark}\label{rem:key}
A convex body $K \subset \R^d$ is in minimal surface area position if its surface area measure (as a measure on the sphere) is isotropic \cite[\S2.3]{ASG15}.
\citeauthor{GP99}~\cite{GP99} show that for convex bodies $K$ in minimal surface area position,
\begin{align*}
\max_{\theta \in S^{d-1}} \left(\frac{\vol_{d-1}(\partial K)}{\vol_{d-1}(P_{\theta}(K))}\right) \leq 2 d\,,
\end{align*}
which is sharp when $K$ is a cube.
Furthermore, for any convex body $K$, there exists a linear bijection $T : \R^d \to \R^d$ such that $TK$ is in minimal surface area position.
Heuristic calculations suggest the calculations in \cref{lem:key} are also sharp when $K$ is a cube, while the technical lemmas used in 
the proof of \cref{thm:main} can likely be improved if all level sets of $\bar f$ are cubes. Exploiting this feels non-trivial, however, and would lead to at most a factor of $d^{1/2}$ improvement
in the final regret bound.
\end{remark}

\begin{proof}[Proof of \cref{thm:main}]
The proof is broken into three parts. 
First we construct the basic exploratory distribution using \cref{lem:key}.
In the second step we argue that only $O(d \log(n))$ exploratory distributions are needed. The final step puts together the pieces using \cref{lem:combine}.
By means of a translation, and without loss of generality, choose coordinates on $\cK$ such that $\zero$ is the minimiser of $\bar f$.

\paragraph{Step 1: Constructing exploratory distributions} 
Let $\epsilon > 0$ and define the level set $K_\epsilon = \{x : \bar f(x) \leq \bar f_\star + \epsilon\}$. 
Let $\cF_\epsilon$ be the set of all $f \in \cF$ for which $f_\star \leq \bar f_\star$ and
\begin{align*}
\Psi_{K_\epsilon}\left(\argmin_{x \in \cK} f(x)\right) \in \left[\frac{1}{128d}, \frac{1}{32d}\right]\,.
\end{align*}
Let $T : \R^d \to \R^d$ be a linear bijection such that $TK_\epsilon$ is in minimal surface area position
and define a probability measure $\rho_\epsilon$ on $\partial K$ by
\begin{align*}
\rho_\epsilon = \frac{\vol_{d-1} \circ T}{\vol_{d-1}(\partial(TK_\epsilon))}\,,
\end{align*}
which is the pullback of the normalised surface area measure on $\partial(TK_\epsilon)$. That is, for any measurable $\varphi : \cK \to \R$,
\begin{align*}
\int_{\partial K_\epsilon} \varphi \d{\rho_\epsilon} = \frac{1}{\vol_{d-1}(\partial(TK_\epsilon))} \int_{\partial(TK_\epsilon)} \varphi \circ T^{-1} \d{\vol_{d-1}}\,.
\end{align*}
Let $f \in \cF_\epsilon$ with minimiser $x_\star \in \cK$.
Since $T$ is a linear bijection, both $f \circ T^{-1}$ and $\bar f \circ T^{-1}$ are convex functions from $\cD = T \cK \to [0,1]$. 
By \cref{lem:invariant} in \cref{sec:technical}, 
\begin{align*}
\Psi_{TK_\epsilon}( Tx_\star ) = \Psi_{K_\epsilon}(x_\star ) \in \left[\frac{1}{128d}, \frac{1}{32d}\right] \,.
\end{align*}
Hence, by \cref{lem:key} and \cref{rem:key}, for any $f \in \cF_\epsilon$,
\begin{align*}
\int_{\partial K_\epsilon} \bar f \d{\rho_{\epsilon}} - f_\star 
&= \frac{1}{\vol_{d-1}(\partial(TK_\epsilon))} \int_{\partial (TK_\epsilon)} \bar f \circ T^{-1} \d{\vol}_{d-1} - f_\star \\
&\leq 2^{13} \sqrt{\frac{2d^3}{\vol_{d-1}(\partial(TK_\epsilon))} \int_{\partial (TK_\epsilon)} (\bar f - f)^2 \circ T^{-1} \d{\vol_{d-1}}} \\
&= 2^{13} \sqrt{2d^3 \int_{\partial K_\epsilon} (\bar f - f)^2 \d{\rho_\epsilon}}\,,
\end{align*}
which shows that $\rho_\epsilon$ satisfies \cref{eq:ass} for all $f \in \cF_\epsilon$.

\paragraph{Step 2: Constructing a partition}
We start by defining a subset $\cF_0 \subset \cF$ on which the Dirac at $\zero$ is a good exploratory distribution. Let
\begin{align*}
\cF_0 = \left\{f \in \cF : \bar f_\star - f_\star \leq 2(\bar f_\star - f(\zero)) \right\} \cup \left\{f \in \cF : f_\star \geq \bar f_\star - 1/n\right\} \,.
\end{align*}
Let $\rho_0$ be a Dirac at $\zero$. Then,
\begin{align*}
\int_\cK \bar f \d{\rho_0} - f_\star 
&= \bar f_\star - f_\star 
\leq \frac{1}{n} + 2 |\bar f_\star - f(\zero)| 
= \frac{1}{n} + \sqrt{4 \int_{\cK} \left(\bar f - f\right)^2 \d{\rho_0}}\,,
\end{align*}
which confirms the claim that $\rho_0$ is a good exploratory distribution for $\cF_0$.
On the other hand, if $f \notin \cF_0$ with minimiser $x_\star \in \cK$, then by the definition of $\cF_0$ and
the assumption that $f$ is $n$-Lipschitz,
\begin{align}
\frac{1}{2n} \leq f(\zero) - f(x_\star) \leq n \norm{x_\star}\,. 
\label{eq:lip}
\end{align}
The above shows that $x_\star \neq \zero$.
Hence, by a continuity argument (\cref{lem:cont}), there exists an $\epsilon > 0$ such that $\Psi_{K_\epsilon}(x_\star) = 1/(64d)$.
The assumption that $\bar f$ is $m$-strongly convex means that $K_\epsilon \subset \{x \in \R^d : \norm{x} \leq \sqrt{2\epsilon / m}\}$ and hence
\begin{align*}
\frac{1}{64d} = \Psi_{K_\epsilon}(x_\star) 
\leq \frac{2\sqrt{2\epsilon/m}}{\norm{x_\star}}
\stackrel{\text{\tiny \cref{eq:lip}}}\leq 4n^2 \sqrt{2\epsilon/m}\,.
\end{align*}
Rearranging shows that $\epsilon \geq m / (2^{17} d^2 n^4) \triangleq \epsilon_0$. 
Let $\gamma = 1 + 1/(9d)$ and $\cE = \epsilon_0\{1, \gamma, \gamma^2,\ldots\} \cap [0,1]$.
Since $\epsilon \geq \epsilon_0$, there exists a $\delta \in \cE$ such that $\delta \in [\epsilon / \gamma, \epsilon]$.
By the convexity of $\bar f$ and the monotonicity of level sets, $\frac{1}{\gamma} K_\epsilon \subset K_{\epsilon / \gamma} \subset K_{\delta} \subset K_\epsilon$. 
Hence, by \cref{cor:monotone},
\begin{align*}
\Psi_{K_\delta}(x_\star) \in \left[\frac{1}{128d}, \frac{1}{32d}\right]\,,
\end{align*}
shows that $f \in \cF_\delta$ and hence $\cF = \cF_0 \cup \bigcup_{\epsilon \in \cE} \cF_\epsilon$.
The previous step demonstrated the existence of an exploratory distribution for each $\cF_\epsilon$ with $\epsilon \in \cE$.

\paragraph{Step 3: Combining}
By definition, $|\cE| \leq 1 + \ceil{\log_\gamma(1/\epsilon_0)}$ with $\gamma = 1 + 1/(9d)$.
Combining \cref{lem:combine} with the exploratory distributions and partitions in steps 1 and 2 completes the proof.
\end{proof}

\section{Technical lemmas}\label{sec:technical}

Here we collect the necessary lemmas. The first few concern properties of $\Psi$: concavity, invariance, approximate monotonicity and continuity.

\paragraph{Concavity of $\Psi$}
The first lemma establishes the concavity of $z \mapsto \Psi_K(x, z)$.
The result is reminiscent of Brunn's concavity principle but relies on a very different proof using the perspective construction.

\begin{lemma}\label{lem:psi-concave}
Let $K \subset \R^d$ be a convex body with $\zero \in K$.
Then, for any $x \notin K$, the function $z \mapsto \Psi_K(x, z)$ is concave on $P_x(K)$.
\end{lemma}

\begin{proof}
Let $\tilde \Psi(z) = \Psi_K(x, z)$.
Parameterise $K = \{z - \alpha x  : f(z) \leq \alpha \leq g(z), z \in P_x(K), \alpha \in \R\}$,
where $f : P_x(K) \to \R$ is convex and $g : P_x(K) \to \R$ is concave (\cref{fig:lem:concave}).
Parameterise the chord connecting $x$ and $z - g(z) x \in \partial K$ by $y(t) = (1 - t) x + t z - t g(z) x$. 
Then,
\begin{align*}
1 - \tilde \Psi(z) 
&= \min\{t \in (0,1] : y(t) \in K\} \\
&= \min\{t \in (0,1] :  f(tz) \leq t + t g(z) - 1\} \\
&= \min\{1/s : s f(z/s) - g(z) + s - 1 \leq 0, s \in [1, \infty)\}\,. 
\end{align*}
Let $h(z, s) = s f(z/s) - g(z) + s - 1$, which is the perspective of $f$ minus a concave function and hence is convex on $P_x(K) \times [1, \infty)$.
Note that $h(z, 1) = f(z) - g(z) \leq 0$, while the fact that $x \notin K$ implies that $\lim_{s \to \infty} h(z, s) = \infty$.
Hence, for any $z, w \in P_x(K)$, there exist largest reals $s$ and $t$ such that $h(z, s) = 0$ and $h(w, t) = 0$ so that $1 - \tilde \Psi(z) = 1/s$ and $1 - \tilde \Psi(w) = 1/t$.
By convexity, $h(\alpha z + (1 - \alpha) w, \alpha s + (1 - \alpha)t) \leq 0$ and therefore the largest value $r$ for which $h(\alpha z + (1 - \alpha) w, r) = 0$ is
at least $\alpha s + (1 - \alpha) t$. Therefore,
\begin{align*}
1 - \tilde \Psi(\alpha z + (1 - \alpha) w)  
= \frac{1}{r} 
&\leq \frac{1}{\alpha s + (1 - \alpha) t}  \\
&\leq \frac{\alpha}{s} + \frac{1-\alpha}{t} \\
&= \alpha (1 - \tilde \Psi(z)) + (1 - \alpha) (1 - \tilde \Psi(w))\,. \qedhere
\end{align*}
\end{proof}

\begin{center}
\includegraphics[width=6cm]{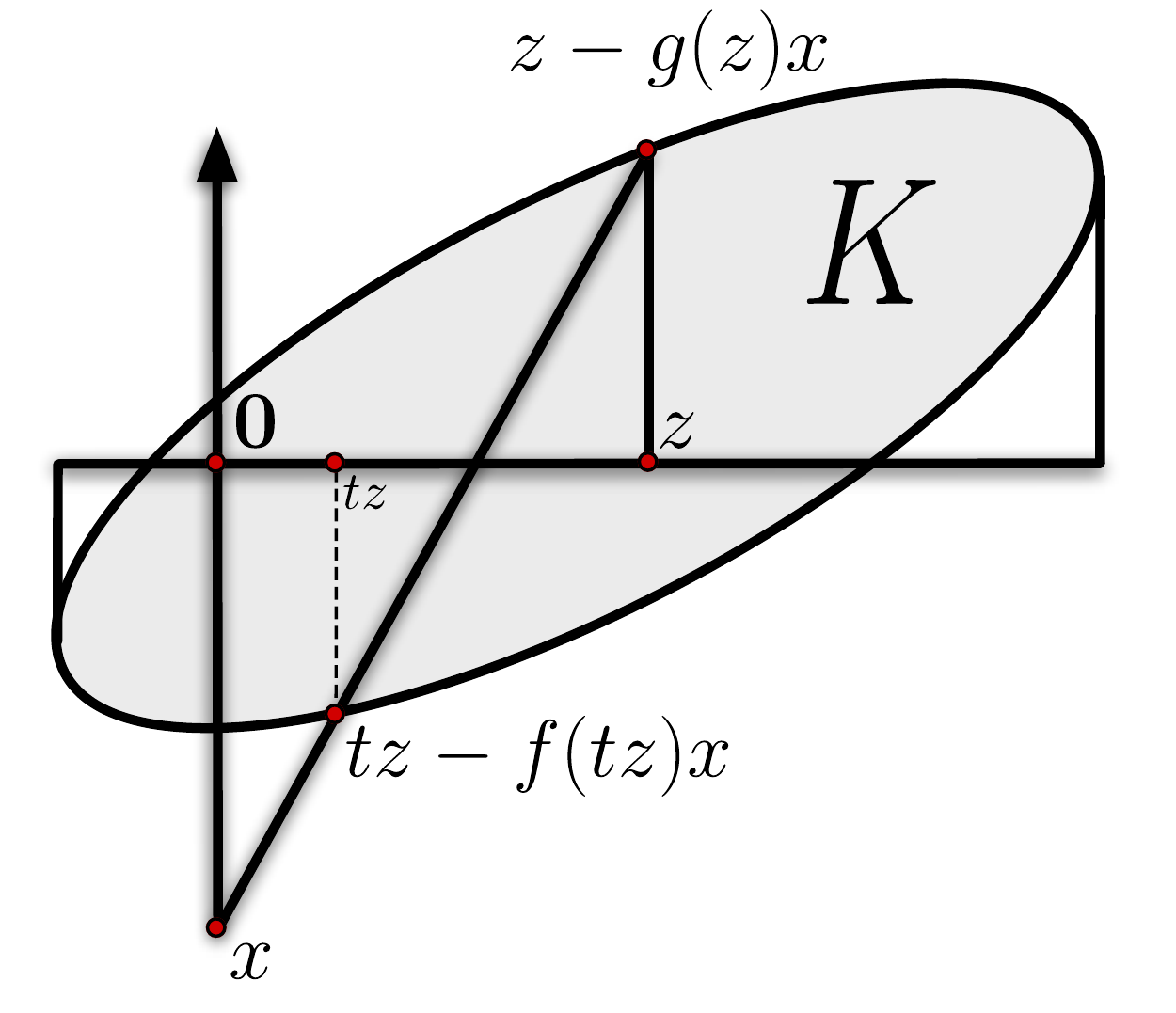}
\captionof{figure}{The construction used in the proof of \cref{lem:psi-concave}.
}\label{fig:lem:concave}
\end{center}

\paragraph{Invariance of $\Psi$}
The next simple lemma shows that $\Psi$ is invariant under linear bijections. 

\begin{lemma}\label{lem:invariant}
Let $T : \R^d \to \R^d$ be a linear bijection, $K \subset \R^d$ be a convex body with $\zero \in K$ and let $x \notin K$. Then
$\Psi_K(x) = \Psi_{TK}(T x)$.
\end{lemma}

\begin{proof}
Let $z \in P_x(K)$ and $\alpha \in \R$ be the smallest value such that $z + \alpha x \in K$.
By definition, $y \triangleq z + \alpha x = \pi_K(x, z)$ and since $z + \alpha x \in K$ if and only if $Tz + \alpha Tx \in TK$, it
follows that 
\begin{align*}
T \pi_K(x, z) = Ty =  Tz + \alpha Tx = \pi_{TK}(Tx, P_{Tx}(Tz)) = \pi_{TK}(Tx, \Lambda(z)) \,,
\end{align*}
where we introduced $\Lambda = P_{Tx} \circ T$.
Thus,
\begin{align*}
\Psi_{TK}(Tx, \Lambda(z)) = \frac{\vol_{d-1}([Tx, Ty] \cap TK)}{\vol_{d-1}([Tx, Ty])} = \frac{\vol_{d-1}([x, y] \cap K)}{\vol_{d-1}([x, y])} = \Psi_K(x, z)\,.
\end{align*}
Then, using the fact that $\Lambda$ is a linear bijection between $P_x(K)$ and $P_{Tx}(TK)$,
\begin{align*}
\Psi_K(x)
&= \frac{1}{\vol_{d-1}(P_x(K))} \int_{P_x(K)} \Psi_K(x, z) \d{\vol_{d-1}(z)} \\
&= \frac{1}{\vol_{d-1}(P_x(K))} \int_{P_x(K)} \Psi_{TK}(Tx, \Lambda(z)) \d{\vol_{d-1}(z)} \\
&= \frac{1}{\vol_{d-1}(P_{Tx}(TK))} \int_{P_{Tx}(TK)} \Psi_{TK}(Tx, w) \d{\vol_{d-1}(w)} \\
&= \Psi_{TK}(Tx)\,. \qedhere
\end{align*}
\end{proof}

\paragraph{Monotonicity and continuity of $\Psi$}
The next few lemmas provide a kind of monotonicity and continuity of the mapping $K \mapsto \Psi_K(x)$, which is used in the proof of \cref{thm:main} to establish the existence of
a suitable grid over level sets.

\begin{lemma}\label{lem:monotone}
Let $A \subset \R^d$ be a convex body with $\zero \in A$ and $\gamma > 1$.
Then $\Psi_{\gamma A}(x) \geq \Psi_A(x)$ for all $x \notin \gamma A$.
\end{lemma}

\begin{proof}
The first step is to argue that $\Psi_{\gamma A}(x, \gamma z) \geq \Psi_{\gamma A}(\gamma x, z)$, which follows, using the notation in \cref{fig:lem:monotone}, because
\begin{align*}
\Psi_{\gamma A}(x, z) = \frac{|y - w|}{|y - x|} = \frac{\norm{y - u}}{\norm{y - \gamma x}} \geq \frac{|y - v|}{|y - \gamma x|} = \Psi_{\gamma A}(\gamma x, z) \,.
\end{align*}
The result follows from the fact that $P_x(\gamma A) = P_{\gamma x}(\gamma A)$ and \cref{lem:invariant}:
\begin{align*}
\Psi_{\gamma A}(x) 
&= \int_{P_x(\gamma A)} \Psi_{\gamma A}(x, z) \d{\vol_{d-1}}(z) \\
&\geq \int_{P_{\gamma x}(\gamma A)} \Psi_{\gamma A}(\gamma x, z) \d{\vol_{d-1}}(z)
= \Psi_{\gamma A}(\gamma x)
\stackrel{\text{\tiny{Lem.\ \ref{lem:invariant}}}}= \Psi_A(x)\,. \qedhere
\end{align*}
\end{proof}

\begin{center}
\includegraphics[width=8cm]{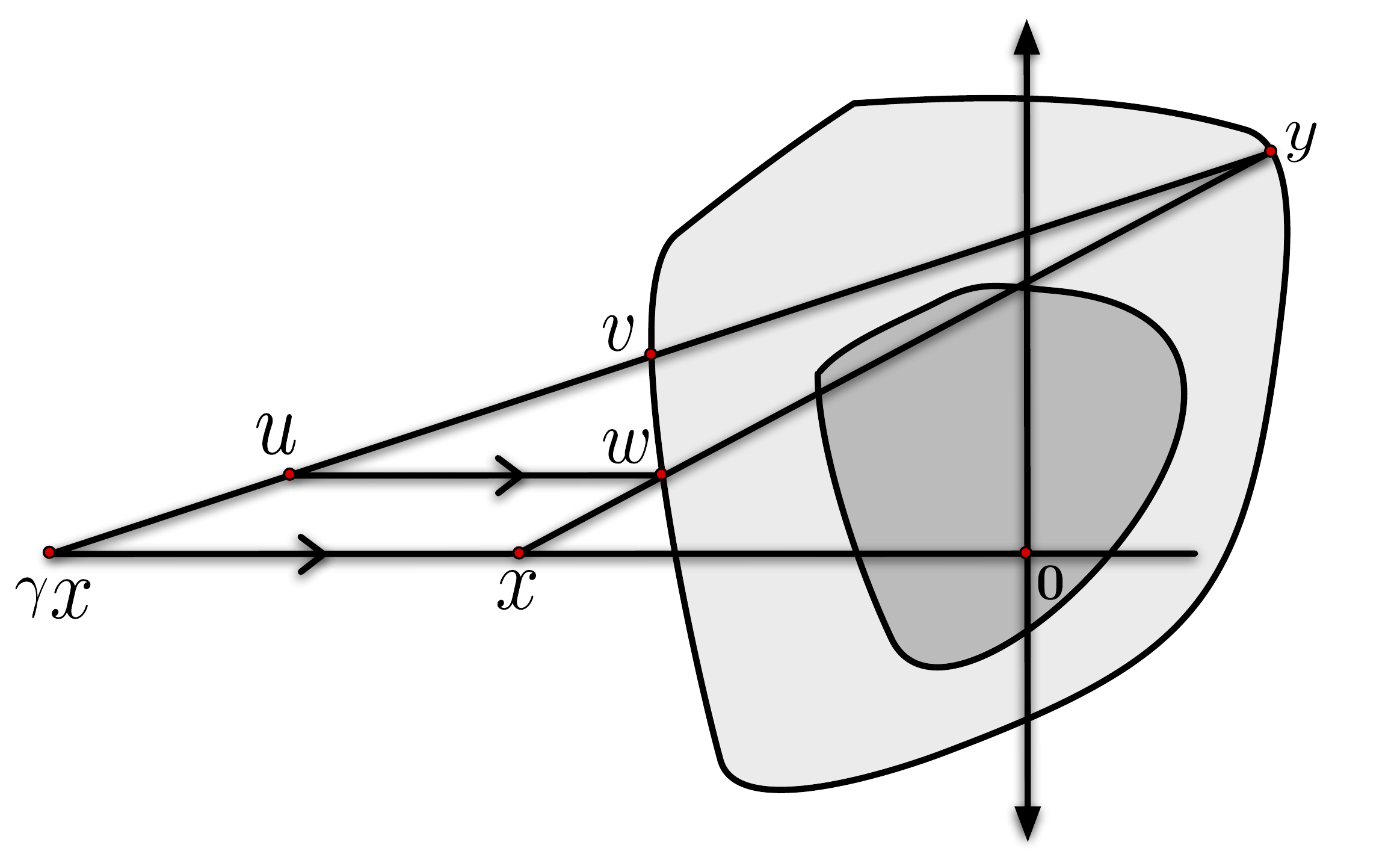}
\captionof{figure}{The construction used in the proof of \cref{lem:monotone}. The point $y$ is $\pi_{\gamma A}(z)$ for some $z \in P_x(\gamma A)$.
The points $v$ and $w$ are chosen so that $[y, \gamma x] \cap \gamma A = [v, y]$ and $[y, x] \cap \gamma A = [y, w]$.
Finally, $u$ is chosen on the chord $[\gamma x, y]$ so that $u - w$ is parallel to $x$.
}\label{fig:lem:monotone}
\end{center}

Recall the definition of $\Psi^\infty_K(x)$ just after \cref{def:psi}.

\begin{lemma}\label{lem:Psi}
Let $A$ and $B$ be convex bodies such that $\zero \in A \subset B \subset \gamma A$ for $\gamma > 1$. 
Assume that $x \notin \gamma A$ and $\Psi^\infty_{\gamma A}(x) \leq 1/2$. 
Then $\Psi_{\gamma A}(x) \leq \left(2\gamma - 1\right)^{d+1} \Psi_B(x)$.
\end{lemma}

\begin{proof}
Let $\Lambda(z) = z / (2 \gamma - 1)$. We claim that for all $z \in P_x(\gamma A)$,
\begin{align}
\Psi_{\gamma A}(x, z) \leq (2 \gamma - 1)^2 \Psi_B(x, \Lambda(z)) \,.
\label{eq:Psi-bound}
\end{align}
Setting the proof of this aside for a moment, the consequence is that
\begin{align*}
\int_{P_x(\gamma A)} \Psi_{\gamma A}(x, z) \d{\vol_{d-1}}(z)
&\leq  (2 \gamma - 1)^2 \int_{P_x(\gamma A)} \Psi_B(x, \Lambda(z)) \d{\vol_{d-1}}(z) \\
&=  (2\gamma - 1)^{d+1}\int_{\Lambda(P_x(\gamma A))} \Psi_B(x, y) \d{\vol_{d-1}}(y) \\
&\leq (2\gamma - 1)^{d+1} \int_{P_x(B)} \Psi_B(x, y) \d{\vol_{d-1}(y)} \,,
\end{align*}
where the last inequality is true because $\Lambda(P_x(\gamma A)) \subset P_x(A) \subset P_x(B)$, since $\gamma/(2\gamma - 1) \leq 1$.
The lemma follows since $\vol_{d-1}(P_x(B)) \leq \vol_{d-1}(P_x(\gamma A))$.
It remains to establish \cref{eq:Psi-bound}, which is a high-school exercise in length chasing. The quantities that follow are defined in the caption of \cref{fig:lem:Psi}.
\begin{align}
\frac{|p - u|}{|p - x|} = \frac{|y - w|}{|y - x|} = \frac{|y - v| |u - q|}{|y - x| |v - q|} = \frac{|u - q|}{|v - q|} \Psi_{\gamma A}(x, z) = \frac{1}{\gamma} \Psi_{\gamma A}(x, z) \,.
\label{eq:p}
\end{align}
Furthermore,
\begin{align*}
\frac{|u - w|}{|y - p|} = 1 - \frac{|p - u|}{|p - x|} = 1 - \frac{1}{\gamma} \Psi_{\gamma A}(x, z) \qquad \quad 
\frac{|u - w|}{|y - q|} = \frac{|v - u|}{|v - q|} = \frac{\gamma - 1}{\gamma}\,.
\end{align*}
Dividing one by the other in the third equality below yields
\begin{align*}
\frac{|r - q|}{|z - q|} 
= \frac{|p - q|}{|y - q|} 
= 1 - \frac{|y - p|}{|y - q|}  
&= 1 - \left(\frac{\gamma - 1}{\gamma}\right) \frac{1}{1 - \Psi_{\gamma A}(x, z) / \gamma} \in \left[\frac{1}{2\gamma - 1}, \frac{1}{\gamma}\right]\,, 
\end{align*}
where the final relation holds because $\Psi^\infty_{\gamma A}(x) \leq 1/2$ by assumption.
Extracting from the above that $|p - q| / |y - q| \leq 1/\gamma$ shows that $p \in A \subset B$.
Combining this with \cref{eq:p} shows there exists a $t \in [r, z]$ such that $\Psi_B(x, t) \geq \Psi_{\gamma A}(x, z) / \gamma$.
Furthermore, the above display also shows that $r \in [\Lambda(z), z]$ and hence, by the concavity of $\Psi_B$,
\begin{align*}
\Psi_B(x, \Lambda(z)) 
&\geq \frac{1}{2 \gamma - 1} \max_{t \in [\Lambda(z), z]} \Psi_B(x, t) \geq \frac{\Psi_{\gamma A}(x, z)}{\gamma(2 \gamma - 1)} \geq \frac{\Psi_{\gamma A}(x, z)}{(2\gamma - 1)^2}\,. \qedhere
\end{align*}
\end{proof}

\begin{center}
\includegraphics[width=8cm]{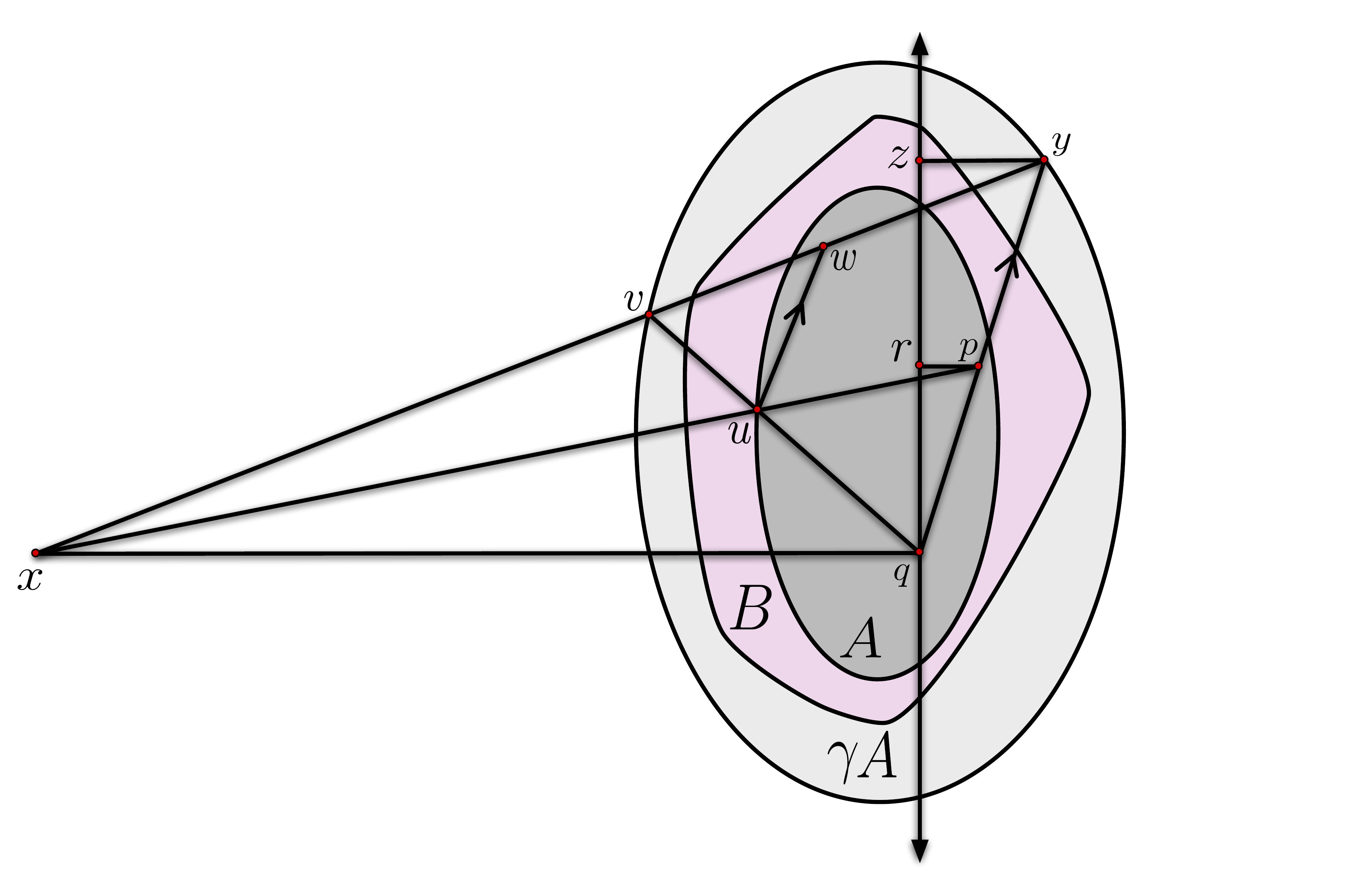}
\captionof{figure}{The construction used in the proof of \cref{lem:Psi}, which
has $q = \zero$, $y = \pi_{\gamma A}(x, z)$. The point $v$ is chosen so that $[x, y] \cap \gamma A = [v, y]$. The point $u$ is such that $[v,q] \cap A = [u,q]$ and $p$
is the intersection of $[q, y]$ and the affine hull $\aff(\{x, u\})$. Lastly, $w$ is the point in $[v,y]$ such that $[u, w]$ is parallel to $[q,y]$ and $r = P_x(p)$.
}\label{fig:lem:Psi}
\end{center}

\begin{corollary}\label{cor:monotone}
Suppose that $A$ and $B$ are convex bodies and $\zero \in \frac{1}{\gamma} A \subset B \subset A$ for $\gamma = 1 + 1/(9d)$. Then, for any $x \notin A$ with $\Psi_A(x) \leq 1/(2d)$,
\begin{align*}
\frac{\Psi_B(x)}{\Psi_{A}(x)} \in [1/2, 2]\,.
\end{align*}
\end{corollary}

\begin{proof}
The second part of \cref{lem:concave-aux} (given next) and concavity of $z \mapsto \Psi_A(x, z)$ shows that $\Psi_A^\infty(x) \leq 1/2$.
By assumption, $\frac{1}{\gamma} A \subset B \subset A$, which implies that $\frac{1}{\gamma} B \subset \frac{1}{\gamma} A \subset B$ and 
so by \cref{lem:monotone,lem:Psi},
\begin{align*}
\left(2 \gamma - 1\right)^{2d+2} \Psi_B(x) 
\tag{\cref{lem:Psi}} &\geq \left(2\gamma - 1\right)^{d+1} \Psi_{A}(x) \\
\tag{\cref{lem:monotone}} &\geq \left(2\gamma - 1\right)^{d+1} \Psi_{A / \gamma}(x) \\
\tag{\cref{lem:Psi}} &\geq \Psi_B(x)\,,
\end{align*}
where the last inequality also uses the fact that $\Psi_B^\infty(x) \leq \Psi_A^\infty(x) \leq 1/2$.
The result follows from the choice of $\gamma$ and naive bounding.
\end{proof}

\begin{corollary}\label{lem:cont}
Suppose that $f \in \cF$ is minimised at $\zero$ and let $K_\epsilon = \{y \in \cK : f(y) \leq f_\star + \epsilon\}$.
Then, for any $x \neq \zero$, there exists an $\epsilon > 0$ such that
\begin{align*}
\Psi_{K_\epsilon}(x) = \frac{1}{64d}\,.
\end{align*}
\end{corollary}

\begin{proof}
That $\epsilon \mapsto \Psi_{K_\epsilon}^\infty(x)$ is continuous and non-decreasing is straightforward.
Using the intermediate value theorem and the facts that $\lim_{\epsilon \to 0} \Psi_{K_\epsilon}(x) = 0$ and that $\Psi_{K_\epsilon}(x) = 1$ when $\epsilon$ is such that $x \in \partial K_\epsilon$,
there exists a smallest value $\epsilon_{\max}$ such that
\begin{align*}
\Psi_{K_{\epsilon_{\max}}}^\infty(x) = 1/64\,.
\end{align*}
By the second part of \cref{lem:concave-aux}, $\Psi_{K_{\epsilon_{\max}}}(x) \geq 1/(64d)$.
Strong convexity of $f$ ensures level sets contract to a point as $\epsilon$ tends to zero and hence, $\lim_{\epsilon \to 0} \Psi_{K_\epsilon}(x) = 0$.
Hence, by the intermediate value theorem it suffices to show that $\epsilon \mapsto \Psi_{K_\epsilon}(x)$ is continuous for $\epsilon \in (0, \epsilon_{\max}]$.
Let $\epsilon \in (0, \epsilon_{\max}]$ and $\gamma > 1$. By convexity of $f$,
$\frac{1}{\gamma} K_{\gamma \epsilon} \subset K_{\epsilon} \subset K_{\gamma \epsilon}$. Repeating the argument in the proof of \cref{cor:monotone} shows that 
$\Psi_{K_\epsilon}(x)$ tends to $\Psi_{K_{\gamma \epsilon}}(x)$ as $\gamma$ tends to $1$. 
\end{proof}

\paragraph{Properties of concave random variables}
The next two lemmas are probably known. They concern the law of a concave random variable under the uniform probability measure on the domain, which
is shown to have constant mass about its expectation.

\begin{lemma}\label{lem:concave-aux}
Let $A \subset \R^{d-1}$ be convex and $\varphi : A \to [0, \infty)$ be concave.
Then,
\begin{align*}
\frac{1}{\vol_{d-1}(A)} \int_A \varphi^2 \d{\vol_{d-1}}
&\leq 2^{5/2} \left(\frac{\int_A \varphi \d{\vol_{d-1}}}{\vol_{d-1}(A)}\right)^2 \,.
\end{align*}
Furthermore, $\max_{x \in A} \varphi(x) \leq \frac{d}{\vol_{d-1}(A)} \int_A \varphi \d{\vol_{d-1}}$.
\end{lemma}

\begin{proof}
Let $B = \{(x, y) : x \in A, \, |y| \leq \varphi(x) \} \subset \R^d$, which is convex.
Define $\theta = (0,\ldots,0,1)$ and $h(t) = \vol_{d-1}(B \cap (\theta^\perp + t \theta))$.
By Brunn's concavity principle \cite[theorem 1.2.1]{ASG15}, $t \mapsto h(t)^{1/(d-1)}$ is concave on its support. By the arithmetic-geometric inequality, for $\lambda \in [0,1]$ and $s, t$ in the support of $h$,
\begin{align*}
h(\lambda s + (1 - \lambda) t)^{\frac{1}{d-1}} \geq \lambda h(s)^{\frac{1}{d-1}} + (1 - \lambda) h(t)^{\frac{1}{d-1}} \geq h(s)^{\frac{\lambda}{d-1}} h(t)^{\frac{1 - \lambda}{d-1}}\,, 
\end{align*}
which implies that $h$ is log-concave.
Then,
\begin{align*}
\frac{1}{\vol_{d-1}(A)} \int_A \varphi^2 \d{\vol_{d-1}}
&= \frac{1}{\vol_{d-1}(A)} \int_B |\ip{x, \theta}| \d{\vol_d}(x) \\
&\leq \frac{1}{\vol_{d-1}(A)} \left(\vol_d(B) \int_B \ip{x, \theta}^2 \d{\vol_d}(x)\right)^{1/2} \\
&= \frac{1}{\vol_{d-1}(A)} \left(\vol_d(B) \int_{-\infty}^\infty t^2 h(t) \d{t}\right)^{1/2} \\
&\leq \frac{1}{\vol_{d-1}(A)} \left(\frac{2\vol_d(B)}{h(0)^2} \left(\int_{-\infty}^\infty h(t) \d{t}\right)^3\right)^{1/2} \\
&= 2^{5/2} \left(\frac{\int_A \varphi \d{\vol_{d-1}}}{\vol_{d-1}(A)}\right)^2 \,,
\end{align*}
where the first inequality follows from Cauchy-Schwarz and the second from corollary 2.24 in the notes by \citeauthor{Tko18} \cite{Tko18} and log-concavity of $h$ and $\int_{-\infty}^\infty t h(t) \d{t} = 0$, which
holds because $h$ is symmetric.
The last equality follows since $h(0) = \vol_{d-1}(A)$ and $\int_{-\infty}^\infty h(t) \d{t} = \vol_d(B) = 2 \int_A \varphi \d{\vol_{d-1}}$.
For the second part, let $x \in \argmax_{y \in A} \varphi(y)$ and $C \subset B$ be the convex hull of $(x, \varphi(x))$, $(x, -\varphi(x))$ and $A \times \{0\}$, which is the union of two cones with the same base.
Then, 
\begin{align*}
\int_A \varphi \d{\vol_{d-1}}
&= \frac{1}{2} \vol_d(B) 
\geq \frac{1}{2} \vol_d(C) 
= \frac{\varphi(x) \vol_{d-1}(A)}{d}\,,
\end{align*}
where the second equality follows from the formula for the volume of a cone.
\end{proof}

\begin{lemma}\label{lem:concave}
Let $A \subset \R^{d-1}$ be convex and $\varphi : A \to [0, \infty)$ be concave and $\nu = \d{\vol_{d-1}} / \vol_{d-1}(A)$ be the uniform probability measure on $A$. Then,
\begin{align*}
\nu\left(\left\{x \in A : \frac{\varphi(x)}{\int_A \varphi \d{\nu}} \in [1/4, 16]\right\}\right) \geq 1/32 \,.
\end{align*}
\end{lemma}

\begin{proof}
By Markov's inequality,
\begin{align*}
\nu\left(\left\{x \in A : \varphi(x) \leq 16 \int_A \varphi \d{\nu}\right\}\right) \geq 1 - \frac{1}{16}\,.
\end{align*}
For non-negative random variable $X$ and any $\theta \in (0,1)$, the Payley--Zygmund inequality says that $\bbP(X \geq \theta \E[X]) \geq (1 - \theta)^2 \E[X]^2/\E[X^2]$.
Combining this with \cref{lem:concave-aux} shows that
\begin{align*}
\nu\left(\left\{x \in A : \varphi(x) \geq \frac{1}{4} \int_A \varphi \d{\nu}\right\}\right) \geq \left(1 - \frac{1}{4}\right)^2 \frac{\left(\int_A \varphi \d{\nu}\right)^2}{\int_A \varphi^2 \d{\nu}} \geq \frac{9}{2^{13/2}} \geq \frac{3}{32}\,.
\end{align*}
Combining the previous displays yields the result.
\end{proof}

\section{Lipschitz and strong convexity relaxation}\label{sec:reduction}

The last ingredient of the proof is to show that
the Lipschitz and strong convexity assumptions are indeed mild.
More or less the same argument has been used elsewhere \cite{BLE17,BE18}.

\begin{proposition}\label{prop:reduce}
\cref{thm:regret} follows from \cref{thm:ids,thm:main}.
\end{proposition}

\begin{proof}
Let $\cF$ be the set of $n$-Lipschitz and $m$-strongly convex functions from convex body $\cJ \subset \R^d$ to $[0,1]$.
\cref{thm:ids,thm:main} show that
\begin{align}
\Reg_n^\star(\cF) \leq \const \cdot d^{2.5} \sqrt{n} \log(n d \diam(\cJ) / m) \,.
\label{eq:J}
\end{align}
\cref{thm:regret} assumes that $\cK$ contains a unit-radius Euclidean ball.
By translation we may assume that $\{x \in \R^d : \norm{x} \leq 1\} \subset \cK$.
Let $(f_t)_{t=1}^n$ be an arbitrary sequence of convex functions from $\cK$ to $[0,1]$, possibly non-Lipschitz and non-strongly convex.
Define $\dist(x, A) = \min_{y \in A} |x - y|$ and $\cJ = \{x \in \cK : \dist(x, \partial \cK) \geq 1/n\}$, which is a convex subset of $\cK$.
Next, let $(\hat f_t)_{t=1}^n$ be the sequence of convex functions from $\cK \to [0,1]$ given by
\begin{align*}
\hat f_t(x) = \left(\frac{n}{n+1}\right) \left(f_t(x) + \frac{1}{n} \left(\frac{\norm{x}}{\diam(\cK)}\right)^2\right) \in [0,1]\,.
\end{align*}
That $\hat f_t$ is $m$-strongly convex with $m = 1/((n+1) \diam(\cK)^2)$ and $\hat f_t(x) \in [0,1]$ is immediate.
To see that $\hat f_t$ is $n$-Lipschitz on $\cJ$, let $g \in \partial \hat f_t(x)$ be a subgradient of $\hat f_t$ for $x \in \cJ$ and let $\alpha > 0$ be such that $x + \alpha g / \norm{g} \in \partial \cK$.
Then, by convexity and boundedness of $\hat f_t$, $1 \geq \hat f_t(x + \alpha g / \norm{g}) - \hat f_t(x) \geq \ip{g, \alpha g / \norm{g}} = \alpha \norm{g}$. 
By the definition of $\cJ$, $\alpha \geq 1/n$ and hence $\norm{g} \leq n$.
Running a policy witnessing \cref{eq:J},
\begin{align}
\const \cdot d^{2.5} &\sqrt{n} \log(n d \diam(\cJ) / m)
\geq \max_{y \in \cJ} \E\left[\sum_{t=1}^n \hat f_t(x_t) - \hat f_t(y)\right] \nonumber \\
&\geq \frac{n}{n+1} \left[\max_{y \in \cJ} \E\left[\sum_{t=1}^n f_t(x_t) - f_t(y)\right] - 1\right] \nonumber \\
&\geq \frac{n}{n+1} \left[\max_{x \in \cK} \E\left[\sum_{t=1}^n f_t(x_t) - f_t(x)\right] - 2\right]\,.  \label{eq:red}
\end{align}
Only the last inequality presents any challenge.
To see why this is true, let $x \in \cK$ be the minimiser of $\sum_{t=1}^n f_t$ on $\cK$ and
let $y = (1 - 1/n) x$.
Since $\{x : \norm{x} \leq 1\} \subset \cK$, it follows that $y \in K$.
Furthermore, since $f_t$ is convex and bounded in $[0,1]$, 
\begin{align*}
f_t(y) 
&\leq (1 - 1/n) f_t(x) + \frac{1}{n} f_t(\zero) \leq f_t(x) + \frac{1}{n}\,. 
\end{align*}
The result follows by rearranging \cref{eq:red}, noting that $\diam(\cJ) \leq \diam(\cK)$ and by substituting the value of $m$.
Note, the dependence on the dimension in the logarithm is dropped, since when $d \geq n$ the regret is better bounded by $n$.
\end{proof}

\appendix

\section{Proof of \cref{thm:ids}}\label{app:ids}

For completeness, we outline the key steps in the proof of \cref{thm:ids}, giving the main arguments and referring to the specific parts
of the literature where necessary. The best reference is the article by \citeauthor{BE18} \cite{BE18}.

\paragraph{Discretisation}
Let $\cC \subset \cK$ be a finite collection of points such that for all $x \in \cK$ there exists a $y \in \cC$ with 
$|x - y| \leq \epsilon \triangleq 1/(n^2 \max(1, (2 \beta)^{1/2}))$.
Standard covering number results \cite[corollary 4.1.14]{ASG15} show that $\cC$ may be chosen so that 
\begin{align}
\log |\cC| \leq d \log\left(\frac{3 \diam(\cK)}{\epsilon}\right)\,. \label{eq:cover}
\end{align}
Let $\Pi(x) = \argmin_{y \in \cC} \norm{x - y}$
and $y^\star = \argmin_{y \in \cC} \sum_{t=1}^n f_t(y)$. Since functions $f \in \cF$ are $n$-Lipschitz, for any policy, 
\begin{align}
\Reg_n = \max_{x^\star \in \cK} \E\left[\sum_{t=1}^n f_t(x_t) - f_t(x^\star)\right] \leq 1 + \E\left[\sum_{t=1}^n f_t(x_t) - f_t(y^\star)\right]\,.
\label{eq:approx}
\end{align}
Exploratory distributions can also be discretised. Given a distribution $\mu$ on $\cF$ and an exploratory distribution $\rho$, let $\xi = \rho \circ \Pi^{-1}$, which
is supported on $\cC$. Then, repeatedly using the assumption that $f \in \cF$ are $n$-Lipschitz and naive bounding,
\begin{align}
\int_\cK \bar f \d{\xi} - \int_\cF f_\star \d{\mu}(f)
&\leq \frac{1}{n} + \int_\cK \bar f \d{\rho} - \int_{\cF} f_\star \d{\mu}(f) \nonumber \\
&\leq \frac{1}{n} + \alpha + \sqrt{\beta \int_{\cF} \int_{\cK} (\bar f - f)^2 \d{\rho} \d{\mu}(f)} \nonumber \\
&\leq \frac{2}{n} + \alpha + \sqrt{2\beta \int_{\cF} \int_{\cK} (\bar f - f)^2 \d{\xi} \d{\mu}(f)} \,. 
\label{eq:disc}
\end{align}
Hence, $\xi$ is a good exploratory distribution with slightly larger constants than $\rho$.

\paragraph{Minimax duality}
Using the assumption that $f \in \cF$ are $n$-Lipschitz, the definition of the discretisation and then a minimax theorem \cite[theorem 1]{LS19pminfo}, we obtain 
\begin{align}
\Reg_n^\star(\cF) 
&\leq 1 + \inf_{\textrm{$\cC$-policies}} \sup_{(f_t)_{t=1}^n \in \cF^n} \E\left[\sum_{t=1}^n f_t(x_t) - f_t(y^\star)\right] \nonumber \\
&= 1 + \sup_{\nu} \inf_{\textrm{$\cC$-policies}} \underbracket{\E\left[\sum_{t=1}^n f_t(x_t) - f_t(y^\star)\right]}_{\BReg_n}\,. \label{eq:minimax} 
\end{align}
where the inf over policies is restricted to policies playing in $\cC$ and the
sup on the right-hand side is over all finitely supported distributions on $\cF$. The expectation on the right-hand side is now over the randomness in both $(x_t)_{t=1}^n$
and $(f_t)_{t=1}^n$, with the latter sampled from $\nu$.
Note, the comparator $y^\star$ is now a random variable.
The expectation on the right-hand side is called the Bayesian regret, denoted by $\BReg_n$, which depends on prior and policy.
For the remainder of this section we fix a prior $\nu$ and design a policy for which the Bayesian regret with respect to $\nu$ is bounded
in terms of $n$, $d$, $\diam(\cK)$ and the constants $\alpha$ and $\beta$ that define the conditions for an exploratory distribution in \cref{thm:ids}.

Note, \citeauthor{BDKP15} \cite{BDKP15} also propose a minimax theorem for convex bandits without a discretisation argument, which inspired the argument for partial monitoring quoted above \cite{LS19pminfo}.
Some details in the former proof are missing, however, such as the choice of topology on the space of convex functions and the continuity of the regret as a function of a deterministic policy.

\paragraph{Posterior, policy and information ratio}
The policy will sample $x_t$ from a probability measure $\xi_t$ on $\cC$, where $\xi_t$ is an exploratory distribution provided by combining
the conditions of \cref{thm:ids} and \cref{eq:disc}.
Let $\E_t$ be the conditional expectation with respect to $(x_s)_{s=1}^t$ and $(f_s(x_s))_{s=1}^t$.
Then, for each $y \in \cC$ and $x \in \cK$, let
\begin{align*}
f_{t,y}(x) &= \E_{t-1}[f_t(x) | y^\star = y] &
\bar f_t(x) &= \E_{t-1}[f_t(x)]\,. 
\end{align*}
Convexity, $m$-strong convexity and $n$-Lipschitzness are preserved under averaging, so $f_{t,y}$ and $\bar f_t$ are in $\cF$ almost surely.
Next, let $\mu_t$ be the finitely supported probability distribution on $\cF$ for which 
$\mu_t(\{f_{t,y}\}) = \E_{t-1}[\sind_{y^\star = y}]$.
By the conditions in \cref{thm:ids} and \cref{eq:disc}, there exists a probability measure $\xi_t$ supported on $\cC$ such that
\begin{align*}
\int_\cK \bar f_t \d{\xi_t} - \int_{\cF} f_\star \d{\mu_t}(f) \leq \frac{2}{n}  + \alpha + \sqrt{2\beta \int_{\cF} \int_{\cK} (\bar f_t(x) - f(x))^2 \d{\xi_t}(x) \d{\mu_t}(f)} \,.
\end{align*}
Then, considering the policy that samples $x_t$ from $\xi_t$ and using \cref{eq:approx}, 
\begin{align*}
\BReg_n
&= \E\left[\sum_{t=1}^n \int_{\cK} \bar f_t \d{\xi_t} - \int_{\cF} f(y^\star) \d{\mu_t}(f)\right] \\
&\leq \E\left[\sum_{t=1}^n \int_{\cK} \bar f_t \d{\xi_t} - \int_{\cF} f_\star \d{\mu_t}(f)\right] \\
&\leq 2 + n \alpha + \E\left[\sum_{t=1}^n \sqrt{2\beta \vphantom{\int_{\cF}} \smash{ \underbracket{\int_{\cF} \int_{\cK} (\bar f_t(x) - f(x))^2 \d{\xi_t}(x) \d{\mu_t}(f)}_{v_t}}}\;\right] \,.
\end{align*}
The final step is the bound
\begin{align*}
\E\left[\sum_{t=1}^n \sqrt{2\beta v_t}\right] 
\leq \sqrt{2n \beta \E\left[\sum_{t=1}^n v_t\right]}
&\leq \sqrt{n \beta \log(|\cC|)}\,,
\end{align*}
where the first inequality is Jensen's and the second follows from Pinsker's inequality to bound the squared difference of expectations in $v_t$ and the chain rule for mutual information \cite[lemma 5]{BDKP15}.
Therefore,
\begin{align*}
\BReg_n \leq 2 + n\alpha + \sqrt{n \beta \log |\cC|}\,. 
\end{align*}
Since this bound is independent of the prior $\nu$, it follows from \cref{eq:minimax} and \cref{eq:cover} that
\begin{align*}
\Reg_n^\star(\cF) \leq 3 + n \alpha + \sqrt{\beta d n \log\left(3n^2 \max(1, (2\beta)^{1/2}) \diam(\cK)\right)}\,.
\end{align*}

\paragraph{Acknowledgements}
Many thanks to my colleagues Andr\'as Gy\"orgy and Marcus Hutter for their useful suggestions.

\bibliographystyle{plainnat}
\bibliography{all}

\begin{thebibliography}{16}
\providecommand{\natexlab}[1]{#1}
\providecommand{\url}[1]{\texttt{#1}}
\expandafter\ifx\csname urlstyle\endcsname\relax
  \providecommand{\doi}[1]{doi: #1}\else
  \providecommand{\doi}{doi: \begingroup \urlstyle{rm}\Url}\fi

\bibitem[Artstein-Avidan et~al.(2015)Artstein-Avidan, Giannopoulos, and
  Milman]{ASG15}
S.~Artstein-Avidan, A.~Giannopoulos, and V.~D. Milman.
\newblock \emph{Asymptotic geometric analysis, Part I}, volume 202.
\newblock American Mathematical Soc., 2015.

\bibitem[Bubeck and Eldan(2018)]{BE18}
S.~Bubeck and R.~Eldan.
\newblock Exploratory distributions for convex functions.
\newblock \emph{Mathematical Statistics and Learning}, 1\penalty0 (1):\penalty0
  73--100, 2018.

\bibitem[Bubeck et~al.(2015)Bubeck, Dekel, Koren, and Peres]{BDKP15}
S.~Bubeck, O.~Dekel, T.~Koren, and Y.~Peres.
\newblock Bandit convex optimization: $\sqrt{T}$ regret in one dimension.
\newblock In \emph{Proceedings of the 28th Conference on Learning Theory},
  pages 266--278, Paris, France, 2015. JMLR.org.

\bibitem[Bubeck et~al.(2017)Bubeck, Lee, and Eldan]{BLE17}
S.~Bubeck, Y-T. Lee, and R.~Eldan.
\newblock Kernel-based methods for bandit convex optimization.
\newblock In \emph{Proceedings of the 49th Annual ACM SIGACT Symposium on
  Theory of Computing}, pages 72--85, 2017.

\bibitem[Cesa-Bianchi and Lugosi(2006)]{Ces06}
N.~Cesa-Bianchi and G.~Lugosi.
\newblock \emph{Prediction, learning, and games}.
\newblock Cambridge University Press, 2006.

\bibitem[Dani et~al.(2008)Dani, Hayes, and Kakade]{DHK08}
V.~Dani, T.~P. Hayes, and S.~M. Kakade.
\newblock Stochastic linear optimization under bandit feedback.
\newblock In \emph{Proceedings of the 21st Conference on Learning Theory},
  pages 355--366, 2008.

\bibitem[Giannopoulos and Papadimitrakis(1999)]{GP99}
A.~Giannopoulos and M.~Papadimitrakis.
\newblock Isotropic surface area measures.
\newblock \emph{Mathematika}, 46\penalty0 (1):\penalty0 1--13, 1999.

\bibitem[Hazan(2016)]{Haz16}
E.~Hazan.
\newblock Introduction to online convex optimization.
\newblock \emph{Foundations and Trends{\textregistered} in Optimization},
  2\penalty0 (3-4):\penalty0 157--325, 2016.

\bibitem[Hazan and Levy(2014)]{HL14}
E.~Hazan and K.~Levy.
\newblock Bandit convex optimization: Towards tight bounds.
\newblock In \emph{Advances in Neural Information Processing Systems}, pages
  784--792, 2014.

\bibitem[Hazan and Li(2016)]{HY16}
E.~Hazan and Y.~Li.
\newblock An optimal algorithm for bandit convex optimization.
\newblock \emph{arXiv preprint arXiv:1603.04350}, 2016.

\bibitem[Hu et~al.(2016)Hu, Prashanth, Gy{\"o}rgy, and Szepesv{\'a}ri]{HPGS16}
X.~Hu, LA. Prashanth, A.~Gy{\"o}rgy, and Cs. Szepesv{\'a}ri.
\newblock (bandit) convex optimization with biased noisy gradient oracles.
\newblock In \emph{Artificial Intelligence and Statistics}, pages 819--828,
  2016.

\bibitem[Lattimore and Szepesv{\'a}ri(2019)]{LS19pminfo}
T.~Lattimore and Cs. Szepesv{\'a}ri.
\newblock An information-theoretic approach to minimax regret in partial
  monitoring.
\newblock In \emph{Proceedings of the 32nd Conference on Learning Theory},
  pages 2111--2139, Phoenix, USA, 2019. PMLR.

\bibitem[Orabona(2019)]{Ora19}
F.~Orabona.
\newblock A modern introduction to online learning.
\newblock \emph{arXiv preprint arXiv:1912.13213}, 2019.

\bibitem[Russo and {Van Roy}(2014)]{RV14}
D.~Russo and B.~{Van Roy}.
\newblock Learning to optimize via information-directed sampling.
\newblock In \emph{Advances in Neural Information Processing Systems}, pages
  1583--1591. Curran Associates, Inc., 2014.

\bibitem[Russo and {Van Roy}(2016)]{RV16}
D.~Russo and B.~{Van Roy}.
\newblock An information-theoretic analysis of {T}hompson sampling.
\newblock \emph{Journal of Machine Learning Research}, 17\penalty0
  (1):\penalty0 2442--2471, 2016.
\newblock ISSN 1532-4435.

\bibitem[Tkocz(2018)]{Tko18}
T.~Tkocz.
\newblock \emph{Asymptotic Convex Geometry Lecture Notes}.
\newblock 2018.

\end{thebibliography}

\end{document}